\documentclass[fontsize=11pt,reqno]{amsart}

\usepackage{amsmath}
\usepackage{amsthm}
\usepackage{amssymb}
\usepackage{mathtools}
\usepackage{mathrsfs}
\usepackage{physics}
\usepackage[margin=1in]{geometry}
\usepackage{scrextend}
\usepackage{tikz-cd}
\usepackage[shortlabels]{enumitem}
\usepackage{colonequals}
\usepackage{hyperref}
\usepackage{comment}
\usepackage{setspace}

\usepackage[notcolon, notquote]{hanging}

\usepackage[title]{appendix}

\setlist[enumerate]{topsep=0pt,itemsep=-0.5ex,partopsep=1ex,parsep=1ex}

\setlist[itemize]{topsep=0pt,itemsep=-0.5ex,partopsep=1ex,parsep=1ex}

\renewcommand{\P}{\mathbb{P}}
\newcommand{\A}{\mathbb{A}}
\newcommand{\C}{\mathbb{C}}
\newcommand{\Z}{\mathbb{Z}}
\newcommand{\Q}{\mathbb{Q}}
\newcommand{\N}{\mathbb{N}}

\newcommand{\I}{\mathdutchcal{I}}
\renewcommand{\O}{\mathdutchcal{O}}
\newcommand{\X}{\mathdutchcal{X}}
\newcommand{\E}{\mathdutchcal{E}}

\newtheorem{thm}{Theorem}[section]
\newtheorem{theorem}[thm]{Theorem}

\newtheorem{lemma}[thm]{Lemma}
\newtheorem{proposition}[thm]{Proposition}
\newtheorem{prop}[thm]{Proposition}

\theoremstyle{definition}
\newtheorem{defn}[thm]{Definition}
\newtheorem{situation}[thm]{Situation}

\theoremstyle{remark}
\newtheorem{remark}[thm]{Remark}
\newtheorem{question}[thm]{Question}
\newtheorem{problem}[thm]{Problem}
\newtheorem{example}[thm]{Example}

\theoremstyle{plain}
\newtheorem{Lthm}{Theorem}

\DeclareMathOperator{\covdeg}{\mathrm{cov.deg}}
\DeclareMathOperator{\cd}{\mathrm{cd}}

\DeclareMathOperator{\cg}{cov.gon}
\DeclareMathOperator{\covgon}{cov.gon}

\newcommand{\struct}[1]{\mathdutchcal{O}_{#1}}
\newcommand{\ot}{\otimes}
\newcommand{\ol}[1]{\overline{#1}}

\newcommand{\iso}{\xrightarrow{\sim}}
\newcommand{\sm}{\setminus}

\newcommand{\wt}[1]{\widetilde{#1}}

\newcommand{\mb}[1]{\mathbb{#1}}

\newcommand{\mtc}[1]{\mathdutchcal{#1}}

\DeclareMathOperator{\Spec}{Spec}

\DeclareMathOperator{\mult}{mult}
\DeclareMathOperator{\Supp}{Supp}

\DeclareMathOperator{\Pic}{Pic}
\DeclareMathOperator{\Hom}{Hom}

\DeclareMathOperator{\ord}{ord}

\DeclareMathOperator{\codim}{codim}

\DeclareMathOperator{\lct}{lct}

\DeclareMathOperator{\LC}{LC}

\DeclareMathOperator{\Sym}{Sym}

\newcommand{\divides}{\mid}

\newcommand{\strusheaf}[1]{\mathdutchcal{O}_{#1}}

\newcommand{\sheafHom}{\mathscr{H}\text{\kern -3pt {\calligra\large om}}\,}

\DeclareMathOperator{\reg}{reg}

\DeclareMathOperator{\divisor}{\mathrm{div}}

\newcommand{\bC}{{\mathbb C}}
\newcommand{\bQ}{{\mathbb Q}}
\newcommand{\bP}{{\mathbb P}}

\DeclareMathOperator{\irr}{irr}

\DeclareMathOperator{\gon}{gon}
\DeclareMathOperator{\cov}{cov.gon}

\DeclareMathOperator{\NKLT}{NKLT}

\usepackage{scalerel}
\usepackage{calligra, mathrsfs}
\DeclareMathOperator{\calHom}{\mathcal{H}\text{om}}

\newcommand{\m}{\mathfrak{m}}

\newcommand\cL{\mathdutchcal{L}}

\newcommand\cW{{\mathcal{W}}}
\newcommand\cC{\mathcal{C}}

\renewcommand{\L}{\mathdutchcal{L}}
\newcommand{\CC}{\mathbb{C}}

\newcommand{\cX}{\mathdutchcal{X}}
\newcommand{\cZ}{\mathdutchcal{Z}}
\newcommand{\ul}[1]{\underline{#1}}
\newcommand{\Mbar}{\ol{\mathcal{M}}} 
\newcommand{\cU}{\mathdutchcal{U}}

\newcommand{\cO}{\mathdutchcal{O}}

\usepackage{float}

\makeatletter
\newcommand{\chapterauthor}[1]{%
  {
    \parindent0pt\vspace*{-10pt}%
    \linespread{1}\scshape by #1%
  }
}
\makeatother

\DeclareMathAlphabet{\mathdutchcal}{U}{dutchcal}{m}{n}
\SetMathAlphabet{\mathdutchcal}{bold}{U}{dutchcal}{b}{n}
\DeclareMathAlphabet{\mathdutchbcal}{U}{dutchcal}{b}{n}

\title{Curves on complete intersections and measures of irrationality}
\author[Nathan Chen, Benjamin Church, and Junyan Zhao]{Nathan Chen, Benjamin Church, and Junyan Zhao \\ with an Appendix by Mohan Swaminathan}


\allowdisplaybreaks

\begin{document}

\begin{abstract}
We study the minimal degrees and gonalities of curves on complete intersections. We prove that the degree of any curve on a general complete intersection $X \subseteq \mathbb{P}^N$ of large multidegree is bounded from below by the degree of $X$. As an application, we answer a problem of Bastianelli--De Poi--Ein--Lazarsfeld--Ullery on measures of irrationality for complete intersections.
\end{abstract}
\maketitle

\begin{spacing}{0.1}
\tableofcontents
\end{spacing}

\section[tocentry = {TEST}]{Introduction}

Given a complex projective variety $X \subseteq \P^N$, it is always covered by curves obtained by taking intersections with linear subspaces. Often it is fruitful to ask: are these the ``simplest'' curves on $X$? More precisely, we ask if there are curves of smaller degree, genus, or gonality. For $X$ a general complete intersection, we make a significant step towards showing there are none. This verifies two open conjectures regarding curves on general complete intersections.

\subsection*{Degrees of curves}
Our first result establishes a folklore conjecture (cf. \cite{Wu90} and \cite[Conjecture 3.10]{Chen24}).

\begin{Lthm} \label{thm:degree_of_any_curve}
Let $X \subseteq \P^{n+r}$ be a general complete intersection variety of dimension $n \geq 1$ cut out by polynomials of degrees $d_{1}, \ldots, d_{r} \geq 2n-1$. Then any curve $C \subseteq X$ satisfies
\[ \deg(C)\ \ge\ (d_1 - 2n+2) \cdots (d_r - 2n+2) .\]
Moreover, there exists\footnote{In fact one can choose $N = 2^{27r(n+r)^{2}} \cdot (n+r)^{6(n+r)}$.} an integer $N \coloneqq N(n,r)$ such that if $d_1, \dots, d_r \ge N$, then 
\[ \deg(C)\ \ge\ d_1 \cdots d_r. \]
\end{Lthm}

\noindent By slicing with hyperplanes, note that the same degree bound appearing in Theorem~\ref{thm:degree_of_any_curve} holds for any positive dimensional subvariety of $X$. For hypersurfaces in $\bP^4$ ($n = 3$ and $r = 1$), the statement of Theorem~\ref{thm:degree_of_any_curve} follows from a theorem of Wu \cite{Wu90}, and the higher dimensional case for hypersurfaces was recently shown by the first author and Yang \cite{CY25}.

The story begins with the classical theorem of Noether--Lefschetz \cite{Lefschetz24}, which states that if $d_{1}+ \cdots + d_{r} \geq r+3$, then a very general complete intersection surface $X \subseteq \P^{r+2}$ of multidegree $(d_{1}, \ldots, d_{r})$ satisfies $\Pic(X) = \Z\langle \O_{X}(1) \rangle$. In particular, this implies that for any curve $C \subseteq X$, its degree $\deg(C)$ is an integer multiple of $d_{1} \cdots d_{r}$. Motivated by the Noether--Lefschetz theorem, Griffiths--Harris later made a series of conjectures concerning curves on very general hypersurfaces in $\P^{4}$ \cite{GH85}. The strongest form of their conjecture was disproved by Voisin \cite{Voisin89}, but one notable case that remains open predicts that on a threefold hypersurface $X \subseteq \P^{4}$ of degree $d \geq 6$, the degree of any curve $C \subseteq X$ satisfies $d \divides \deg{C}$.  Significant progress in this direction was initiated by Koll\'{a}r \cite{Kollar90}, who used degeneration methods to give counterexamples to the integral Hodge conjecture for smooth hypersurfaces. More recently, Paulsen refined these degeneration methods to show that for a positive density set of integers, the conjecture of Griffiths--Harris (and its generalization to higher-dimensional hypersurfaces) is true \cite{Paulsen21}. 

\subsection*{Measures of irrationality}

Another motivation for studying curves on varieties is as a measure of complexity. In recent years, measures of irrationality have emerged as interesting and subtle birational invariants for projective varieties, generalizing the notion of gonality for curves to higher dimensions. For a projective variety $X$ of dimension $n$, the \emph{degree of irrationality} and the \emph{covering gonality} are defined as follows:
\[ \irr(X)\ \coloneqq\ \min\big\{\delta>0\ |\ \exists\textup{ rational dominant map } X\dashrightarrow \mb{P}^n\textup{ of degree }\delta\big\}; \]
\[ \cov(X)\ \coloneqq\ \min\big\{c>0\ |\ \exists\textup{ a curve of gonality } c \textup{ through a general point } x\in X\big\}.\]
From their descriptions, we see that the degree of irrationality is a measure of how far $X$ is from being rational, while the covering gonality is a measure of how far $X$ is from being uniruled. These two quantities are related by: $\irr(X) \geq \cg(X)$. For an overview of the subject, we refer the reader to \cite{BDELU17} and the many references therein.

Given a smooth projective non-degenerate variety $X\subseteq\mb{P}^{n+r}$ of dimension $n>0$ and degree $d$, an elementary observation is that by generic projection from points on $X$, we always have $\irr(X) \leq d-r$. The most challenging difficulty in computing $\irr(X)$ is often finding obstructions to the existence of low degree maps $X \dashrightarrow \P^{n}$. For very general hypersurfaces of large degree (i.e. $r = 1$), we now have a fairly complete picture. In \cite{BCFS18}, the authors computed their covering gonality. Around the same time, \cite{BCD14, BDELU17} proved that if $X$ is a very general hypersurface of sufficiently large degree, then $\irr(X)=d-1$; in fact, they also showed that any dominant map $X\dashrightarrow \mb{P}^n$ of degree $d-1$ is given by projection from a point on $X$. The main idea that appears in these papers is to use the positivity of the canonical bundle $K_{X}$. Note that some of these results concerning the covering gonality have been extended to positive characteristic \cite{Smith20}.

For a (very) general complete intersection variety $X\subseteq \mb{P}^{n+r}$ cut out by polynomials of degrees $d_{1}, \ldots, d_{r}$, it is natural to ask how both invariants behave. The methods of \cite{BDELU17} only yield lower bounds that are additive in the $d_{i}$'s. However, the authors of \emph{loc}.\ \emph{cit}.\ conjectured that the degree of irrationality should behave \textit{multiplicatively} in the degrees, namely there should be a positive constant $C(n,r)$ depending only on $n$ and $r$ such that
\begin{equation}\label{eq:multbound}
\irr(X) \ \geq\ \cg(X) \ \geq\ C(n,r) \, d_1\cdots d_r,
\end{equation}
for sufficiently large degrees $d_{i}$ (see \cite[Problem 4.1]{BDELU17}). From the previous paragraph, it follows that $\irr(X) \leq d_{1} \cdots d_{r} - r$ so we must always have $C(n,r) < 1$. 
\par 
Our second result fully resolves this problem by exhibiting asymptotically sharp bounds:

\begin{Lthm} \label{thm:main_gonality}
Let $X \subseteq \P^{n+r}$ be a general complete intersection variety of dimension $n$ cut out by polynomials of degrees $d_{1}, \ldots, d_{r} \ge n$.  Then
\[ \covgon(X) \ge (d_1 - 2 (n+1)  \sqrt{d_1})(d_2 - n + 1) \cdots (d_r - n + 1) + 1. \]
\end{Lthm}

\noindent 
In particular, for any $\epsilon > 0$, there is an integer $N(\epsilon ; n , r)$ such that when $d_i \ge N(\epsilon ; n, r)$ we have 
\[ \covgon(X) \ge (1 - \epsilon) d_1 \cdots d_r. \]
Notably, Theorem~\ref{thm:main_gonality} gives covering gonality bounds for \textit{general} complete intersections and not just very general ones, as opposed to some of the results in \cite{BDELU17} and \cite{BCFS18}. In particular, the covering gonality bounds hold for almost all complete intersections defined over $\Q$.

The first evidence for a bound of the form \eqref{eq:multbound} comes from a theorem of Lazarsfeld from the 1990s, where vector bundle techniques were used to establish sharp multiplicative bounds for complete intersection curves \cite[Exercise 4.12]{Laz97} (this was later revisited in \cite{HLU20}). Unfortunately, these methods do not seem to generalize to higher dimensions. For $n \geq 2$, Stapleton derived superadditive bounds for the covering gonality of very general codimension two complete intersections \cite[Theorem 5.4]{Sta17}. Subsequently, Stapleton and Ullery computed the degree of irrationality for codimension two complete intersections of type $(2, d)$ and $(3, d)$ \cite[Theorem A \& B]{SU20}. By relating the covering gonality of $X$ to conjectures about the degrees and genera of curves on $X$, the first author produced multiplicative bounds for codimension two complete intersections ($r=2$) and complete intersection surfaces ($n=2$), but the constants were far from sharp. Most recently, the authors of \cite{LSU23} establish sharp multiplicative bounds for the degree of irrationality of complete intersections whose degrees are sufficiently spread out (i.e. $d_1\gg d_2\gg\cdots\gg d_r\gg0$).

\subsection*{Separation of many points}

An important ingredient in the proof of Theorem~\ref{thm:main_gonality} which may be of independent interest is an Angehrn--Siu type result (cf.\ \cite{AS95}) about separation of points for adjoint line bundles:

\begin{Lthm}[Multi-point separation by adjoint line bundles]\label{thm:AS_type_separation_of_points}
Let $(X, H)$ be a polarized normal Gorenstein variety. Suppose there exists a nonempty open subset $U \subseteq X^{\reg}$ of the regular locus and a number $\alpha > 0$ such that any curve $C \subseteq X$ meeting $U$ satisfies $\deg_H{C} \ge \alpha$. Then there exists a constant $\delta \coloneqq \delta(X, H)$ such that the linear series $|K_X + d H|$ separates at least $(d - \delta \, \sqrt{d}) \cdot \alpha$ distinct points on $U$. Moreover, if $H$ is very ample, then one can take $\delta = 2\dim{X}$.
\end{Lthm}

This technical strengthening greatly improves the reduction step that appears in \cite{Sta17, Chen24}. Our approach gives different bounds than those arising from the main result of \cite{AS95}, see Remark~\ref{rmk:compareAS}, in that we optimize their techniques for separating asymptotically large numbers of points. Notably, unlike the assumptions in \cite[Corollary~0.4]{AS95}, the intersection-theoretic criterion above only involves curves and not higher-dimensional subvarieties. Theorem~\ref{thm:AS_type_separation_of_points} gives the optimal leading term for $d \gg 0$ (e.g.\ for $X$ a general complete intersection as illustrated by Theorem~\ref{thm:main_gonality}) but does not usually give the optimal error term. For example, a Reider-like argument \cite{Rei88} shows that when $X$ is a surface, the adjoint linear series $|K_X + d H|$ separates at least $\alpha (d - \alpha / H^2) - 1$ points, where $\alpha \coloneqq \min_{C} (C \cdot H)$ is the minimal degree of an effective curve $C \subset X$. This has an $O(1)$ error term as opposed to $O(\sqrt{d})$ as given in Theorem~\ref{thm:AS_type_separation_of_points}.

\subsection*{Auxiliary results on bounds of other invariants}

Theorem~\ref{thm:degree_of_any_curve} can be used in concert with results in \cite{Clemens86, Ein88, Xu94, Voisin96, Pacienza04, CR04} and others to give an absolute lower bound on the \emph{geometric genus} of any subvariety on a general complete intersection. The result for curves is recorded in \S\ref{section:genus_bounds}.
\par 
In \S\ref{section:association}, we apply Theorem~\ref{thm:AS_type_separation_of_points} to bound the \emph{degree of association} between complete intersections, an invariant introduced in \cite{LM23} to measure how similar the birational types of two varieties of the same dimension are. 
We remark in \S\ref{section:generalizations} that the degeneration method used in the proof of Theorem~\ref{thm:degree_of_any_curve} does not rely on the geometry of projective space in any way. Hence we are able to obtain results for complete intersections in any ambient space, possibly formed by intersecting different ample classes.  

\subsection*{Sketch of proofs and outline of the paper}

Both Theorem~\ref{thm:degree_of_any_curve} and Theorem~\ref{thm:main_gonality} follow from a sharp lower bound on the degree of a curve passing through a \textit{general} point on a general complete intersection $X$ of sufficiently large degrees given in Theorem~\ref{thm:main_covering_degree}. Moreover, our technique provides an explicit multiplicative bound for the covering degree and computes
\[ \covdeg(X, \O_{X}(1))\ =\ \deg{X} \] 
(see Definition~\ref{def:covdeg}) in the large degree regime. The proof of Theorem~\ref{thm:main_covering_degree} proceeds by degenerating a complete intersection $X$ into a union of two complete intersections $X_{1} \cup X_{2}$ of smaller degrees meeting transversally along a complete intersection of one dimension less. Fixing a curve $C$ on the generic such $X$, if we could show that the specialization of $C$ has a component in each $X_{i}$, then we could apply induction to conclude. Unfortunately, this naive approach cannot work -- consider the case of lines on a quadric hypersurface degenerating to a union of two hyperplanes. However, when the curve $C$ belongs to a \emph{covering family of curves} (see \S 2 for the precise definition), we show that either the curves break into components covering each $X_{i}$, or some component must cover the intersection $X_1 \cap X_2$. This allows us to run a more subtle induction simultaneously on degrees, dimension, and codimension.

A crucial input in the argument above is a result on the structure of curves on the reducible central fiber of a flat family that deform to nearby fibers. This follows from ideas of Jun Li \cite{Li01}, and we give a self-contained exposition in Appendix~\ref{appendix}. The Grassmannian technique developed by Reidl and Yang \cite{RY22} then allows one to pass between the degree of \textit{any} curve on a general complete intersection and the covering degree of a general complete intersection of the same multidegree but twice the dimension (see \hyperref[proof:thmA]{Proof of Theorem}~\ref{thm:degree_of_any_curve}). Hence, Theorem~\ref{thm:main_gonality} can be used to give similar bounds for the gonality of any curve on a general complete intersection.

In \S 2, we set up the machinery for degenerating stable maps when the ambient space degenerates to the union of two irreducible components meeting transversally and we exhibit 
constraints on which curves can deform to the general fiber. In \S 3, we prove a lower bound on the covering degree of curves on a general complete intersection $X\subseteq \mb{P}^{N}$ with respect to $\O_{X}(1)$, and use this to prove Theorem~\ref{thm:degree_of_any_curve}. In \S 4 we set up the relevant definitions of multiplier ideals, and we collect the necessary lemmas for modifying their log canonical centers. In \S 5, we then use these lemmas to prove Theorem~\ref{thm:AS_type_separation_of_points}; at the end of the section, we also give the proof of Theorem~\ref{thm:main_gonality}. Finally, in \S 6 we present a number of applications and questions that arise from our work.

\subsection*{Conventions}

In this paper, we work over $\mb{C}$. By \emph{variety}, we mean an integral separated scheme of finite type over $\CC$. A \emph{curve} is a $1$-dimensional geometrically reduced and connected separated scheme of finite type over a field (note that a reducible curve is not a variety according to our conventions). A \emph{family of curves} $\pi \colon \cC \rightarrow T$ is a flat proper morphism whose fibers are curves. A \emph{polarized} variety $(X, L)$ consists of a projective variety $X$ together with an ample line bundle $L$. Recall that the \emph{gonality} of a smooth projective irreducible curve $C$, denoted by $\gon(C)$, is the smallest degree of a finite morphism $C\rightarrow \mb{P}^1$. For an integral curve, we define the gonality to be the gonality of its normalization.

For a Cartier divisor $D$, we denote by $\big|D\big|_{\mb{Q}}$ the $\mb{Q}$-linear series, which consists of effective $\mb{Q}$-divisors which are $\mb{Q}$-linearly equivalent to $D$. For additional properties about linear series and line bundles, we mainly follow \cite{Lazarsfeld04a}. By a collection of \emph{general points} $p_{1}, \ldots, p_{m}$ in $X$, we mean that each $p_{i}$ is chosen arbitrarily and independently from some unspecified open subset of $X$. This should be distinguished from a collection of points $p_1, \dots, p_m$ in \textit{general position}, by which we mean $(p_1, \dots, p_m)$ is chosen from an open subset of $\Sym^{m}(X)$. For example, a linear series separating $n$ \emph{points in general position} is an elementary consequence of the dimension of the linear series whereas separating $n$ \emph{general points} is the more restrictive and subtle property of the linear series studied in \S\ref{section:covering_gonality_mult_ideals}-\ref{section:separation_of_points}.
\par
When working with complete intersections $X \subseteq \P^{n + r}$, we will use $n = \dim{X}$ to denote the dimension and $r$ to denote the codimension. Sometimes we will write $N = n + r$ and $X \subseteq \P^N$. We say that $X$ is a complete intersection of \emph{multidegree} $\underline{d} = (d_{1}, \ldots, d_{r})$ if it is pure dimension $n$ and defined by polynomials of degrees $d_1, \dots, d_r$.

\subsection*{Acknowledgments} 

We thank \.{I}zzet Co\c{s}kun, Mohammad Farajzadeh-Tehrani, Fran\c{c}ois Greer, Melissa Liu, Rob Lazarsfeld, and David Stapleton for the many insightful discussions and encouragement. The second author would particularly like to thank his advisor, Ravi Vakil, for invaluable guidance and advice during this project. The third author would like to thank Lawrence Ein for suggesting Angehrn-Siu's approach for obtaining a stronger separation of points statement. Finally, we would like to thank the referees for all of their helpful suggestions and comments. During the preparation of this article, NC was partially supported by an NSF postdoctoral fellowship DMS-2103099 and BC was partially supported by an NSF GRFP fellowship under grant DGE-2146755.

\section{Covering families and degenerations of stable maps}

For a projective variety $X$ equipped with an embedding $X \hookrightarrow \P^N$, we write $\Mbar_g(X, b)$ for the Kontsevich moduli space of stable maps from genus $g$ curves which have degree $b$ with respect to the embedding. Likewise, if $\X \to T$ is a flat projective morphism equipped with a relatively ample line bundle $\L$, we write $\Mbar_g(\X / T, b)$ for the relative moduli space of stable maps of degree $b$ with respect to $\L$. We sometimes omit the degree $b$. For details and conventions regarding stable maps, we refer to \cite{FP97, Kontsevich95}.

\begin{defn} \label{def:deforms}
    Let $\X \to S$ be a morphism of schemes, $s_0 \in S$ a point, and $\mu_{s_0} : C \to \X_{s_0}$ be a stable map. For any point $s\in S$, we say that $\mu_{s_0}$ \textit{deforms to $\X_{s}$} if there exists a family of stable maps
    \begin{center}
        \begin{tikzcd}
        \cC_{T} \arrow[r, "\mu"] \arrow[d] & \X \arrow[d]
        \\
        (T, t_0) \arrow[r] & (S, s_0)
        \end{tikzcd}
    \end{center}
    such that $T$ is connected and $s$ is in the image of $T \to S$ and $(\mu_{s_0})_{\kappa(t_0)} \cong \mu_{t_0}$. 
\end{defn}

\noindent Note that if $s_0$ is in
the closure of the scheme-theoretic point $s$, we can always take $T$ to be the spectrum of a DVR.

\begin{defn}
    A \textit{covering family of curves} on an irreducible projective scheme $X$ consists of a smooth projective family
    \[ \pi \colon \cC \rightarrow T \]
    of curves parametrized by an irreducible quasi-projective variety $T$, together with a dominant morphism (meaning it hits the generic point of each component of $X$)
    \[ f \colon \cC \rightarrow X, \]
    such that for general $t \in T$, the map $f_{t} \colon C_{t} \rightarrow X$ is birational onto its image.
\end{defn}

\begin{defn}\label{def:covdeg}
    Let $(X, L)$ be an irreducible quasi-projective polarized variety. The \textit{covering degree} of $(X, L)$, denoted by $\covdeg(X, L)$ (or $\covdeg_L(X)$ to de-emphasize $L$), is the minimal integer $b$ such that there exists a covering family of curves $\{ \pi \colon \cC \to T, f \colon \cC \rightarrow X \}$ with
    \[ \deg f^{\ast}L|_{\cC_t} = b. \]
\end{defn}

\begin{remark}\label{rem:defcovdeg}
Alternatively, by standard compactification arguments (cf.\ \cite[Prop 2.6]{Vis89}) one can define the \textit{covering degree} $\covdeg_L(X)$ as the minimal $d \in \Z_{\geq 1}$ such that there exists a family of stable curves $\cC \to T$ over an irreducible base scheme $T$ and a \emph{surjective} stable map $f : \cC \to X$ such that $\deg{f^* L|_{\cC_t}} = b$. Since $X$ is irreducible, one can reduce to checking the degrees of stable maps where the general curve $C_{t}$ is smooth and irreducible.
\end{remark}

By a Hilbert scheme argument, it is straightforward to show that covering degree is lower semi-continuous in flat polarized families. In particular, it is a constructible function and does not increase under specialization. It will be useful to set up notation to refer to the covering degree of a complete intersection of fixed degrees:

\begin{defn}
For $n,r \in \Z_{\geq 1}$ and $d_1, \dots, d_r \in \Z_{\geq 1}$ we define
\[ \cd_{n,r}(d_1, \dots, d_r) \coloneqq \max_{(X_{d_1}, \dots, X_{d_r})} \covdeg(X_{d_1} \cap \cdots \cap X_{d_r}, \O(1)), \]
where $X_{d_i} \subseteq \P^{n+r}$ is a hypersurface of degree $d_i$ and the maximum is taken over tuples $(X_{d_1}, \dots, X_{d_r})$  such that the intersection is smooth of dimension $n$. 
\end{defn}

\noindent
By the lower semi-continuity of covering degree, it is immediate that $\cd_{n,r}(d_1, \dots, d_r)$ can be computed as $\covdeg(X_{d_1} \cap \cdots \cap X_{d_r})$ for a general complete intersection $X_{d_1} \cap \cdots \cap X_{d_r} \subseteq \P^{n+r}$.

\subsection{The Breaking Lemma} 

The proof of Theorem \ref{thm:main_covering_degree} in \S\ref{sec:covering degree of complete intersections} proceeds by induction and relies on constraints governing which curves on the central fiber of an SNC degeneration of varieties deform to nearby fibers. These results rely on multiplicity matching conditions along the boundary that are well-known to experts in relative Gromov--Witten theory. 

\begin{defn}\label{definition:snc_degeneration}
Let $R$ be a DVR, and $s,\eta \in \Spec{R}$ be the closed point and generic point, respectively. An \textit{SNC degeneration of varieties} over $R$ is a flat proper family $f : \X \to \Spec{R}$ such that $\X_\eta$ is a smooth variety and $\X_s$ is reduced with simple normal crossing (SNC) singularities.
\end{defn}

It will be convenient to label certain types of components of a stable map.

\begin{defn}\label{defn:types}
Let $ X_1 \cup_Z X_2$ be the union of two smooth irreducible varieties along a smooth divisor $Z$, and $\mu : C \to X_1 \cup_Z X_2$ be a stable map. A sub-curve $C' \subseteq C$ (connected, but not necessarily irreducible) is said to be of
\begin{enumerate}
    \item \emph{ghost type} if $\mu(C')$ is a point in $Z$;
    \item \emph{type $Z$} if it is not of ghost type and $\mu(C')\subseteq Z$; 
    \item \emph{type $X_i$} (for $i=1$ or $2$) if it is neither of ghost type nor of type $Z$, and $\mu(C')\subseteq X_i$.
\end{enumerate}
\end{defn}

To fix notation throughout this section, we will work in the following situation.

\begin{situation} \label{situation:breaking_degeneration}
Let $R$ be a DVR and $s,\eta \in \Spec{R}$ be the closed point and generic point, respectively. Let $f : \X \to \Spec{R}$ be a SNC degeneration of varieties such that $\X_s = X_1 \cup_Z X_2$ is the union of two smooth irreducible varieties along a smooth divisor $Z$. 
\end{situation}

The following lemma is the critical input that allows us to force certain curves to break into reducible curves whose image has components lying in both components of the special fiber.

\begin{lemma} \label{lemma:divisorial_multiplicity_matching}
In Situation~\ref{situation:breaking_degeneration}, suppose there is a non-constant stable map $\mu_s : C \to \X_s$ that deforms to $\X_\eta$. Let $F \subseteq C$ be a connected component of $\mu_s^{-1}(Z)$,  which is contracted to a point $z \in \X^{\reg}$ of the regular locus of the total space. For $i=1,2$, let $C_i$ be the union of all components of $C$ of type $X_i$ meeting $F$.
Then 
%
\[ \sum_{p \in F \cap C_1} m_p(C_1 ; Z) \ =\ \sum_{p \in F \cap C_2} m_p(C_2 ; Z), \]
where $m_p(C_i ; Z)$ is the multiplicity at which $C_i$ intersects $Z$ in $X_i$ at the point $p$, for $i=1,2$.
\end{lemma}

\begin{remark}
    In the case that $F$ is a point, the only claim above is that $F$ is a node lying on the intersection of two curves $C_1, C_2$, where $C_i$ is of type $X_i$ with multiplicity $m_F(C_1 ; Z) = m_F(C_2 ; Z)$.
\end{remark}

The only consequence of Lemma \ref{lemma:divisorial_multiplicity_matching} needed in the proof of Theorem~\ref{thm:main_covering_degree} is that there exists at least one component of type $X_1$ meeting $F$, and at least one component of type $X_2$ meeting $F$. This result follows from Jun Li's relative stable maps formalism \cite{Li01}. We refer to Appendix~\ref{appendix} for a self-contained proof and for a brief comparison with related statements in the literature. A consequence of this result is the following lemma.

\begin{lemma} \label{lemma:breaking}
In situation~\ref{situation:breaking_degeneration}, let $W \subseteq \X_s$ be the singular locus of the total space. Suppose $\mu_s : C \to \X_s$ is a nonconstant stable map that deforms to $\X_\eta$, and $z \in Z \sm W$ is a point in the image of $\mu$. Then one of the following holds:
\begin{enumerate}[\normalfont(a)]
    \item $z$ lies on the image of a component of type $Z$; or
    \item $z$ lies on the image of a component of type $X_1$ and also on the image of a component of type $X_2$.
\end{enumerate}
\end{lemma}

\begin{proof}
Since $\mu \colon C \to \X_s$ is non-constant, there must be a component $C' \subseteq C$ meeting $\mu^{-1}(z)$ that is either of type $Z$ or (without loss of generality) of type $X_1$. In the former situation we arrive at case (a), so let us now assume that $C'$ is of type $X_1$. Let $p \in C'$ be a point mapping to $z$. If $p$ lies on a component of type $Z$, then we are in case (a); otherwise $p$ is contained in a connected component $F$ of $\mu_s^{-1}(Z)$ such that $\mu(F)=z$. Since $z \in Z \sm W$, we can apply Lemma~\ref{lemma:divisorial_multiplicity_matching} to conclude that $F$ meets a component $C_1$ of type $X_1$ and a component $C_2$ of type $X_2$. Hence the images of $C_1$ and $C_2$ contain $z$ (since they meet $F$ and $\mu(F) = z$) and satisfy the conditions of case (b).
\end{proof}

\begin{theorem} \label{thm:deform_and_break_degree}
In situation~\ref{situation:breaking_degeneration}, let $\L$ be a line bundle on $\X$, which is relatively ample over $\Spec R$. Suppose that $\X^{\reg} \cap Z$ is nonempty, and that $\covdeg(\X_{\ol{\eta}}, \L_{\ol{\eta}}) \le d$. Then either
\begin{enumerate}[\normalfont(a)]
\item $\covdeg(X_1, \L|_{X_1}) + \covdeg(X_2, \L|_{X_2}) \le d$, or 
\item $\covdeg(Z, \L|_Z) \le d$.
\end{enumerate}
\end{theorem}

\begin{proof}
The idea is that through a general point $z \in Z$, there is a curve in the specialization of the covering family passing through $z$. Then we apply Lemma~\ref{lemma:breaking} to conclude that either $z$ lies on a component of type $Z$ or it lies on two components, one of type $X_1$ and one of type $X_2$. Since $z$ is a general point, in the former case, the specialization of the covering family contains a covering family of curves of $Z$ so we conclude (b), in the latter case, there is a component covering $X_1$ and a component covering $X_2$ so we conclude (a). We now fill in the details of this argument.

By assumption, there is a family of stable curves $\pi : \cC \to T$ over an integral base $T$ and a stable dominant morphism $\mu : \cC \to \X_{\ol{\eta}}$ such that $\deg \mu^* \L |_{\cC_{t}} \le d$ fiberwise. By universality of the moduli stack, this gives a morphism $T \to \Mbar_{g}(\X_{\eta})$. Let $\cW \subseteq \Mbar_{g}(\X / R)$ be the stack theoretic closure of the image of $T$. By \cite[Proposition 2.6]{Vis89}, one can choose a finite cover $W \to \cW$ by an integral scheme. 
Since $W$ is integral and dominates $\Spec{R}$, it is flat over $\Spec{R}$. Therefore, pulling back to $W$ we have the following diagram
\begin{center}
\begin{tikzcd}
\cC \arrow[r, "\mu"] \arrow[d, "\pi"'] & \X \arrow[d]
\\
W \arrow[r] & \Spec{R},
\end{tikzcd}
\end{center}
where $\mu$ is surjective because it is surjective over $\eta$, the family $\cC$ is proper over $R$, and $\X$ is irreducible. Specializing to the geometric special fiber, we get a stable map
\begin{center}
\begin{tikzcd}
\cC_s \arrow[r, "\mu_s"] \arrow[d, "\pi_s"'] & \X_s
\\
W_s
\end{tikzcd}
\end{center}
with $\mu_s : \cC_s \to \X_s$ surjective and $\deg \mu^* \L |_{\cC_s} \le d$ by flatness. Since $X_1$ and $X_2$ are both irreducible, there exists an irreducible component 
$S \subseteq W_s$ such that
$Z \subseteq X_i \subseteq \mu(\mathcal{C}_S)$
for some $i \in \{1,2\}$.

Let $\eta_Z \in Z$ be the generic point. Since $\mu$ is surjective, there is a point $\delta \in \cC_S$ mapping to $\eta_Z$. Let $S' \subseteq S$ be the image under $\pi$ of an irreducible component of $\mu^{-1}_s(Z)$ containing $\delta$. By construction, every fiber of $\cC_{S'}\rightarrow S'$ has image meeting $Z$ and $\mu_s(\cC_{S'})$ contains $Z$. Let $\eta_{S'} \in S'$ be the generic point. By the flatness of $\pi : \cC_{S'} \to S'$, the irreducible components of $\cC_{S'}$ are identified with those of $\cC_{\eta_{S'}}$.  Since $\mu : \cC_{\eta_{S'}} \to \X_s$ is a stable map that deforms to $\X_{\eta}$ (there is a specialization from the generic point of $W$ to $\eta_{S'}$ since $W$ is irreducible), we may apply Lemma~\ref{lemma:breaking} to the stable map $\mu_s : \cC_{\eta_{S'}} \to \X_s$, where we take $z = \eta_Z$. Then either there is a component $\cC_Z \subseteq \cC_{\eta_{S'}}$ with $\cC_Z$ of type $Z$ and hitting $\eta_Z$, in which case $\cC_Z \to S'$ is a covering family of $Z$ so we are in case (b), or there are two irreducible components $\cC_i \subseteq \cC_{S'}$ such that $(\cC_i)_{\eta_{S'}}$ is of type $X_i$ for $i=1,2$, and $\eta_{Z}$ lies on both of their images. For the latter, since $\mu(\cC_i)$ and $X_i$ are all irreducible and $\mu(\cC_i)$ properly contains $Z$, we must have that $\mu(\cC_i) = X_i$. Therefore, $\cC_{S'} \to S'$ is a covering family of $X_1 \cup_Z X_2$ over an irreducible base $S'$, and we have that
\begin{align*}
d \ \ge\ \deg_{T} \mu^{\ast}\L \ =\ \deg_{S'} \mu_{\ast}\L |_{\cC_{S'}} \ & \ge\ \deg_{S'} \mu^{\ast}\L |_{\cC_1} + \deg_{S'} \mu^{\ast}\L |_{\cC_2} \ \\
& \ge\ \covdeg(X_1, \L|_{X_1}) + \covdeg(X_2, \L|_{X_2}),
\end{align*}
where the subscript indicates that the degree is taken over $S'$ (as in the line bundles live on curves over a base scheme). This yields case (a).
\end{proof}

\section{Covering degree of complete intersections}\label{sec:covering degree of complete intersections}

In this section, we prove the following lower bound on the covering degree of complete intersections and show how it implies the bounds of Theorem~\ref{thm:degree_of_any_curve} on the degree of \emph{any} curve.

\begin{theorem} \label{thm:main_covering_degree}
Let $X \subseteq \P^{n+r}$ be a general complete intersection variety of dimension $n$ cut out by polynomials of degrees $d_{1}, \ldots, d_{r} \ge n$. Then
\[ \covdeg(X, \O(1)) \ge (d_1 - n + 1) \cdots (d_r - n + 1). \]
Moreover, there exists $N \coloneqq N(n,r)$ such that if $d_1, \dots, d_r \ge N$ then
\[ \covdeg(X, \O(1)) = d_1 \cdots d_r. \]
\end{theorem}

\begin{remark}
When $X \subseteq \P^{n+1}$ is a hypersurface of degree $d = n$, then $X$ is covered by lines so $\covdeg(X) = 1$. However, a basic dimension count shows that when $d > n$, the general such hypersurface is not covered by lines and hence $\covdeg(X) \ge 2$. This aligns with the constants in Theorem~\ref{thm:main_covering_degree}. Moreover, if $d \leq 2n-1$ then every hypersurface contains a line but if $d \geq 2n$ then the general such hypersurface does not, aligning with the constants in Theorem~\ref{thm:degree_of_any_curve}. 
\end{remark}

\subsection{The Main Induction}

To prove this result, we will degenerate our complete intersection to the union of complete intersections of lower degrees and track how stable maps deform using the results of the previous section. This leads to the recursion below, which enables a simultaneous induction on the dimension and codimension of $X$.

\begin{thm} \label{thm:inductive_bound}
Let $d_1 = a + b$ for positive integers $a, b$. Then for any $d_2, \dots, d_r \in \Z_{\geq 1}$ we have
\begin{align*}
\cd_{n,r}(d_1, \dots, d_r) &\ge \min \{ \cd_{n,r}(a, d_2, \dots, d_r) + \cd_{n,r}(b, d_2, \dots, d_r), \\
& \hspace{1.1cm} \cd_{n-1,r+1}(a,b,d_2, \dots, d_r) \}.
\end{align*}
\end{thm}

\begin{proof}
Given a very general hypersurface $X_{d_1}$ in $\mb{P}^{n+r}$, there exists a degeneration to a union $X_a \cup X_b$ of two very general hypersurfaces. More explicitly, we can find equations $F,G,H$ of degrees $d_1, a, b$ such that $V(G)$ and $V(H)$ are isomorphic to the geometric generic hypersurface of degrees $a$ and $b$, respectively, and $V(F), V(G)$, and $V(H)$ all meet transversally. Choosing other fixed hypersurfaces $X_{d_2}, \dots, X_{d_r}$ meeting each other and the aforementioned hypersurfaces all transversally, the family
\[ \X \ \coloneqq\  V(GH - t F) \cap X_{d_2} \cap \cdots \cap X_{d_r} \  \subseteq \ \P^{n+r} \times \A^1 \]
satisfies the hypotheses in Theorem~\ref{thm:deform_and_break_degree}, and the geometric generic fiber of $\X$ is isomorphic to the geometric generic complete intersection of multidegree $\underline{d}$ (cf. \cite[Lemma 2.1]{Vial13}). Let $Z$ denote the intersection of the two components of the special fiber:
\[ Z \colonequals V(G) \cap V(H) \cap X_{d_2} \cap \cdots \cap X_{d_r}. \]
Note that the singular locus $W \subseteq \X$ is the intersection
\[ W \colonequals V(F) \cap V(G) \cap V(H) \cap X_{d_2} \cap \cdots \cap X_{d_r}, \]
which is a divisor in $|\O_{Z}(d)|$; in particular, $W$ does not contain $Z$ as is required in Theorem~\ref{thm:deform_and_break_degree}.
Let $R = \CC[t]_{(t)}$ and $K = \CC(t)$. By the previous discussion,
\begin{enumerate}
\item $\cd_{n,r}(d_1, \dots, d_r) = \covdeg({\X}_{\bar{\eta}}, \O(1))$;
\item $\cd_{n,r}(a, d_2, \dots, d_r) = \covdeg(X_1, \O(1))$;
\item $\cd_{n,r}(b, d_2, \dots, d_r) = \covdeg(X_2, \O(1))$;
\item $\cd_{n-1, r+1}(a,b, d_2, \dots, d_r) = \covdeg(Z, \O(1))$.
\end{enumerate}
Therefore, we may apply Theorem~\ref{thm:deform_and_break_degree} to conclude.
\end{proof}

The following will serve as the base case of an induction argument.

\begin{example}\label{Example:fanoindex1}
Let $X \colonequals X_{\underline{d}} \subseteq \P^{n+r}$ be a general Fano index 1 complete intersection, i.e. $\sum d_{i} = n+r$. It is well known that the Fano variety of lines on $X$ has dimension equal to the expected dimension $2(n+r)-\sum d_{i} - r - 2 = n-2$, so the locus of lines sweeps out at most a divisor on $X$. This implies that the covering degree of $X$ is at least 2, and it turns out that conics do indeed cover $X$. 
\end{example}

We are now ready to prove the first statement in Theorem~\ref{thm:main_covering_degree}.

\begin{theorem} \label{thm:explicit_covering_degree}
For any fixed $n,r$, and any $\underline{d}$ with $d_{i} \geq n$ for all $i$, we have
\begin{equation} \label{eq:degree_inequality}
\cd_{n,r}(d_1, \dots, d_r)\  \geq\  \prod_{i=1}^{r} (d_{i} - n+1).
\end{equation}
\end{theorem}

\begin{proof} 
We apply induction on both $n$ and $r$ simultaneously. The case $n = 1$ of curves is clear. The case $r = 0$ is exactly the fact that $\covdeg(\P^n) = 1$ since in this case the product in equation~(\ref{eq:degree_inequality}) is empty. Towards induction, assume the statement holds for $(n, r-1)$ and $(n-1, r)$; we need to prove it holds for $(n,r)$. Splitting $d_1 = (d_1 - 1) + 1$, we obtain using Theorem~\ref{thm:inductive_bound} that
\begin{align*}
     &\cd_{n,r}(d_1, \dots, d_r) \\
     &\ge \min \big\{ \cd_{n,r}(d_{1}-1, d_2, \dots, d_r) + \cd_{n,r}(1, d_2, \dots, d_r), \cd_{n-1,r+1}(d_{1}-1,1,d_2, \dots, d_r) \big\}
     \\
     &= \min \big\{ \cd_{n,r}(d_{1}-1, d_2, \dots, d_r) + \cd_{n,r-1}(d_2, \dots, d_r), \cd_{n-1,r}(d_{1}-1,d_2, \dots, d_r) \big\}.
\end{align*}

Now we may apply induction on the degrees to the first term in the expression above. For the second term, we use the inductive hypothesis $(n-1,r)$ to conclude
\begin{align*}
\cd_{n-1,r}(d_{1}-1,d_2, \dots, d_r) &\geq ((d_{1}-1) - (n-1)+1) \cdot \prod_{i=2}^{r} (d_{i} - (n-1) + 1) \\
&\ge \prod_{i=1}^{r} (d_{i} - n+1).
\end{align*}
We use the following base case $d_1 = n$ for induction on $d_1$:
\begin{align*}
\cd_{n,r}(n, d_2, \dots, d_r) \ge \cd_{n,r}(1, d_2, \dots, d_r) &= \cd_{n, r-1}(d_2, \dots, d_r) \\
&\ge \prod_{i=2}^r (d_i - n + 1) = \prod_{i = 1}^{r} (d_i - n + 1),
\end{align*}
where the first inequality follows from generic projection or from specialization to a union of complete intersections of multidegrees $(1, d_{2}, \ldots, d_{r})$. Assuming the inequality for $d_1 - 1$, we conclude that
\begin{align*}
&\cd_{n,r}(d_{1}-1, d_2, \dots, d_r) + \cd_{n,r-1}(d_2, \dots, d_r) \\
&\geq ((d_{1}-1) - n+1) \cdot \prod_{i=2}^{r} (d_{i}-n+1) + \cd_{n,r-1}(d_2, \dots, d_r) \\
&\geq (d_{1} - n) \cdot \prod_{i=2}^{r} (d_{i}-n+1) + \prod_{i=2}^{r} (d_{i}-n+1) = (d_{1} - n+1) \cdot \prod_{i=2}^{r} (d_{i}-n+1). \qedhere
\end{align*} 
\end{proof}


\subsection{Proof of Theorem~\ref{thm:degree_of_any_curve}} \label{section:proof_of_AC}

In this section, we complete the proof of Theorem~\ref{thm:degree_of_any_curve}. First we will prove the remaining claims of Theorem~\ref{thm:main_covering_degree} using the following result of Paulsen, the asymptotics resulting from Theorem~\ref{thm:explicit_covering_degree}, and the main inductive bound. Then we will use the Grassmanian techniques of \cite{RY22} to reduced Theorem~\ref{thm:degree_of_any_curve} to Theorem~\ref{thm:main_covering_degree}.

\begin{theorem}[{\cite[Proposition 7]{Paulsen21}}]\label{thm:Paulsen}
Let $n \ge 3$ be an integer. If $d$ is an integer such that
\begin{enumerate}
\item $d$ is coprime to $n!$
\item the largest prime power $q$ dividing $d$ satisfies
\[ \left( \binom{n}{2} - 1 \right) \cdot q^n + \left( n! - \binom{n}{2} \right) \cdot q^{n-1} + (2^n + 1) \cdot n! \le d ,\]
\end{enumerate}
then every curve $C$ on a very general hypersurface $X_d \subseteq \P^{n+1}$ satisfies $d \divides \deg{C}$.
\end{theorem}

Using this result, we complete the proof of Theorem~\ref{thm:main_covering_degree} as follows.

\begin{proof}[Proof of Theorem~\ref{thm:main_covering_degree}]

It remains to show that for fixed $(n,r)$, there is an integer $N \coloneqq N(n,r)$ such that for all $d_1, \dots, d_r \ge N$ we have $\cd_{n,r}(d_1, \dots, d_r) = d_1 \cdots d_r$. To do this, we will check that the set $S_n$ of positive integers $d$ satisfying the hypotheses of Theorem~\ref{thm:Paulsen} satisfies the following condition:

\begin{enumerate}
\item[$(\ast)$]
\begin{center}
For any positive integers $k$ and $r$, there exists an integer $N \coloneqq N(n,k,r)$ such that for any $d_1, \dots, d_r \ge N$ there exists a finite array $\{ a_j^i \}_{i,j}$ of elements $a_j^i \in S$ such that
\begin{enumerate}
\item $a^i_j \ge k$ for any $i,j$;
\item any two elements $a^i_j$ and $a^{i'}_{j'}$ are coprime whenever $j \neq j'$;
\item for all $1 \le j \le r$ we have
\[ d_j = \sum_i a^i_j. \]
\end{enumerate}
\end{center}
\end{enumerate}
\noindent
Granting that $S_n$ satisfies $(\ast)$, let us prove the result. Fixing $(n,r)$, choose $k > \max\{ n, 5 \}$ so that $\cd_{n-1,r+1}(a_1, \dots, a_{r+1}) \ge (1 - \tfrac{1}{2}) a_1 \cdots a_{r+1}$ whenever $a_1, \dots, a_{r+1} \ge k$. Such $k$ exists by an application of Theorem~\ref{thm:explicit_covering_degree}). This condition implies that if we split $d_1 = a + b$ for $a,b \ge k$ and $d_2, \dots, d_r \ge k$ then
\begin{align*}
\cd_{n,r}(d_1, \dots, d_r) & \ge \min \{ \cd_{n,r}(a, d_2, \dots, d_r) + \cd_{n,r}(b, d_2, \dots, d_r), \\
& \hspace{1.1cm} \cd_{n-1, r+1}(a,b, d_2, \dots, d_r) \} 
\\
& \ge \cd_{n,r}(a, d_2, \dots, d_r) + \cd_{n,r}(b, d_2, \dots, d_r).
\end{align*}
The second inequality follows from the fact that the second term in the ``min'' expression is automatically larger than $ \tfrac{1}{2} ab d_2 \cdots d_r \ge d_1 \cdots d_r$ since $\tfrac{1}{2} ab \ge a + b$ for $a,b > 5$. By permutation symmetry in the inputs, $\cd_{n,r}$ is a multi-superadditive function $(\N_{\ge k})^r \to \N$ meaning that
\[ \cd_{n,r}\left( \sum_{i_1} d_{1,i_1}, \dots, \sum_{i_r} d_{r, i_r} \right) \ge \sum_{i_1, \dots, i_r} \cd_{n,r}(d_{1,i_1}, \dots, d_{r, i_r}) \]
provided its inputs $d_{j,i_j}$ are all at least $k$. From now on, let $k(n,r)$ denote the minimum value of $k$ allowed in the preceding argument.
\par
Using property $(\ast)$, for any $d_1, \dots, d_r \ge N(n, k(n,r), r)$ we can find a matrix $\{ a^i_j \}$ satisfying (a) and (b) so that $d_j = \sum_i a^i_j$. Since all $a^i_j \ge k(n,r)$, using the multi-superadditivity we get
\[ \cd_{n,r}(d_1, \dots, d_r) \ge \sum_{i_1, \dots, i_r} \cd_{n,r}(a_1^{i_1}, \cdots, a_r^{i_r}). \]
Since $a_1^{i_1}, \dots, a^{i_r}_r$ are elements of $S_n$, any curve $C$ on a very general complete intersection $X \subseteq \P^{n+r}$ of multidegree $(a^1_{i_1}, \dots, a^1_{i_r})$ satisfies $a^{i_j}_j \divides \deg{C}$. Furthermore, because $a_1^{i_1}, \dots, a^{i_r}_r$ are pairwise coprime, $a^1_1 \cdots a^1_r \divides \deg{C}$. Therefore,
\[ \cd_{n,r}(d_1, \dots, d_r) \ \ge \sum_{i_1, \dots, i_r} a_1^{i_1} \cdots a_r^{i_r} \ = d_1 \dots d_r. \]

Now we prove that $S_n$ satisfies $(\ast)$. We first claim that $S_n$ satisfies $(\ast)$ if it contains arbitrarily long sequences of pairwise coprime integers. Indeed, let $g_1, \dots, g_r, g'_1, \dots, g'_r$ be such a sequence of length $2r$ with all entries $\ge k$. Then we claim that the requisite matrices can be built so that $a_j^i$ is either $g_j$ or $g_j'$ for each $i$. Clearly, such a matrix satisfies (a) and (b) so it suffices to show that all sufficiently large sequences $(d_1, \dots, d_r)$ are representable. This uses nothing more than the claim that if $g,g'$ are coprime then any $d > (g-1)(g'-1)$ can be written as $g x + g' y$ for $x,y > 0$, which is  Sylvester's answer to the well-known ``postage stamp'' or ``coin problem.''
\par
Finally, we show that $S_n$ contains arbitrarily long sequences of coprime integers. There is a constant $C_n$ (depending only on $n$ and given explicitly in Theorem~\ref{thm:Paulsen}) such that given an increasing sequence of primes $p_1, \dots, p_s$ such that $p_1 > 2^n$ and
\begin{equation} \label{equation:requisite_inequality}
    C_n p_{s}^n \le p_1 \cdots p_s,
\end{equation}
the product $p_1 \cdots p_s$ is a member of $S_n$. Indeed, $\gcd(p_1 \cdots p_{s}, n!) = 1$ because $p_s > \dots > p_1 > 2^n > n$ so, together with the inequality (\ref{equation:requisite_inequality}), $d = p_1 \cdots p_s$ satisfies the hypotheses (1) and (2) of Theorem~\ref{thm:Paulsen}.
By Bertrand's postulate, we can choose an arbitrarily long sequence $p_1, \dots, p_s$ such that $p_i \le 2^i p_1$ for $i = 1, \dots, s$. Therefore, if
\[ 2^{ns} C_n \, p_1^n \le p_1^s \]
then the inequality (\ref{equation:requisite_inequality}) will be satisfied and thus $p_1 \cdots, p_s \in S_n$. Since $p_1 > 2^n$, the bound $2^{ns} C_n \, p_1^n \le p_1^s$ holds when $s$ is sufficiently large. By starting at the next largest prime after $p_{s}$ and using the same procedure, we can construct a new element of $S_n$ sharing no primes in common with $p_1 \cdots p_s$. Repeating this produces an arbitrarily long sequence of coprime elements of $S_n$. If we track how large $s$ and $p_1$ need to be, one can get a uniform bound of $N = 2^{27r(n+r)^{2}} \cdot (n+r)^{6(n+r)}$.
\end{proof}

For completeness, here we provide the details of how the Grassmanian technique of \cite{RY22} applies to reduce Theorem~\ref{thm:degree_of_any_curve} to Theorem~\ref{thm:main_covering_degree}. 

\begin{proof}[Proof of Theorem~\ref{thm:degree_of_any_curve}] \label{proof:thmA}
We will apply \cite[Theorem 2.3]{RY22} to the locus of points on the universal complete intersection through which passes a curve of degree $< b \coloneqq \cd_{2n-1, r}(d_1, \dots, d_r)$. We have shown that
\[ \cd_{2n-1, r}(d_1, \dots, d_r) \ge (d_1 - 2n + 2) \cdots (d_r - 2n + 2), \]
where $d_i \ge 2n-1$ for all $i$.
Borrowing from the notation of \cite{RY22}, let $\cX_{n, \ul{d}}$ be the universal smooth complete intersection in $\P^{n+r}$ of type $\ul{d} = (d_1, \dots, d_r)$, and let $\cZ^{<b}_{n, \ul{d}} \subseteq \cX_{n, \ul{d}}$ be the space of pairs $(p, [X])$ consisting of a complete intersection $X$ and a point $p \in X$ through which there exists an integral curve $p \in C \subseteq X$ with $\deg{C} < b$.  Since the arithmetic genus of an integral curve in $\P^{n+r}$ of degree $< d$ is bounded above by $\frac{1}{2}(d-1)(d-2)$ (one can see this by projecting to a $\P^{2}$), there are only finitely many Hilbert polynomials that can appear so the Hilbert scheme of such curves is of finite type. Hence, $\cZ^{<b}_{m,\ul{d}}$ is constructible for all $m$. By definition, $\cZ^{<b}_{2n-1, \ul{d}}$ is a proper subset of $\cX_{n, \ul{d}}$. Hence if $m = 2n$ then $\cZ^{<b}_{m-1, \ul{d}}$ has codimension at least $1$ in $\cX_{m-1, \ul{d}}$. Furthermore, if $(p, [X]) \in \cZ^{<b}_{n', \ul{d}}$ and $X = \varphi^{-1}(X' \cap H)$ where $\varphi : \P^{n'+r} \iso H$ is a parametrized hyperplane and $X'$ is a complete intersection of type $\ul{d}$ in $\P^{n'+1+r}$ then $(\varphi(p), X') \in \cZ^{<b}_{n'+1, \ul{d}}$ because $\varphi(C) \subseteq X'$ is a curve through $\varphi(p)$ of degree $< d$. The hypotheses in \cite[Theorem 2.3]{RY22} are satisfied, so we conclude that for every integer $c \geq 1$,
\[ \codim \big(\cZ^{<b}_{m-c, \ul{d}}\subseteq \cX_{m-c, \ul{d}}\big) \ \ge\  c. \]
Since $n = m - n$ and $\cX_{n, \ul{d}}$ has relative dimension $n$ over the moduli space of complete intersections, by setting $c = n$ we conclude that $\cZ^{<b}_{n, \ul{d}}$ has codimension $\ge n$ inside the universal family. However, by definition, the fibers of the second projection from $\cZ^{<b}_{n, \ul{d}}$ are covered by curves and hence $\cX_{n, \ul{d}}$ cannot intersect the general fiber. This completes the proof.
\end{proof}

\section{Covering gonality and multiplier ideals} \label{section:covering_gonality_mult_ideals}

\subsection{Covering Gonality and Separation of Points} 

In this subsection, we recall how positivity of the canonical line bundle yields lower bounds on the covering gonality.

\begin{defn}
Let $X$ be a normal projective variety. For a positive integer $m$, we say that a line bundle $L$ on $X$ \textit{separates $m$ distinct general points} if there exists a Zariski-open $U \subseteq X$ such that for every set of $m$ distinct points $\xi \colonequals \{ p_{1}, \ldots, p_{m} \} \subseteq U$, the restriction map
\[ H^{0}(X, L) \rightarrow H^{0}(X, L \ot \cO_{\xi}) \]
is surjective. If the choice of $U$ is made explicit we say that $L$ \textit{separates $m$ distinct points on $U$}.
\end{defn}

The relevant result we need is the following:

\begin{proposition}[{\cite[Theorem 1.10]{BDELU17} \& \cite[Remark 5.14]{Sta17}}]\label{separatepoints}
    Let $X$ be a normal Gorenstein variety with at worst canonical singularities. If the canonical line bundle $K_X$ separates $m$ distinct general points, then
    \[ \cov(X)\geq m+1. \]
\end{proposition}

\subsection{Singularities, multiplier ideals and vanishing theorems}

In this subsection, we will introduce the relevant multiplier ideal machinery (for more details, see \cite{Ein97, Lazarsfeld04b}).

Although most of our results hold in a more general setting, for simplicity we will assume that $X$ is $\bQ$-Gorenstein.

\begin{defn}\textup{
    Let $(X,D)$ projective log pair such that $D$ is an effective $\mb{Q}$-Cartier $\mb{Q}$-divisor, and fix a log resolution $\mu : X' \rightarrow (X,D)$. Then the \emph{multiplier ideal sheaf}
    \[ \mtc{J}(D)\ =\ \mtc{J}(X,D)\ \subseteq \ \mtc{O}_X \]
    associated to $D$ is defined to be
    \[ \mtc{J}(D)\ \coloneqq\ \mu_{*}\mtc{O}_{X'}(K_{X'}-\lfloor \mu^{*}(K_X + D) \rfloor) \]
    where $\lfloor E \rfloor$ denotes the round-down of a $\bQ$-divisor $E$. Let $Z(X,D)$ (or $Z(D)$ if there is no confusion about the ambient space) denote the scheme defined by the multiplier ideal $\mtc{J}(D)$.
}\end{defn}

\begin{defn}\label{nonklt}\textup{
    Let $(X,D)$ be a log canonical pair. A prime divisor $E$ over $X$ is called a \emph{log canonical place} or \textit{lc place} of $(X,D)$ if there exists a birational model (or equivalently, a log resolution) $\pi:X'\rightarrow X$ of $(X,D)$ such that
    \[ \ord_{E}\big(\pi^{*}(K_X+D)-K_{X'}\big)\ =\ 1. \]
    The closed subset $\pi(E)$ of $X$ is called an \emph{lc center} of $(X,D)$. The \emph{non-klt locus} of $(X,D)$, denoted by $\NKLT(X,D)$, is the union of all lc centers of $(X,D)$.
}\end{defn}

\begin{defn}
Consider a pair $(X, D)$ where $D$ is an effective $\bQ$-Cartier divisor. Given $x\in X$, the \emph{LC-locus} of $D$ at $x$, denoted by $\LC_x(X,D)$, is the union of all irreducible components of $Z(X, c_x \cdot D)$ that pass through $x$, with a reduced scheme structure, where
\[ c_x \  =\  \lct_x(X, D) \ \colonequals\  \inf \{ c' > 0 \mid \mtc{J}(c' \cdot D) \text{ is nontrivial at } x \} \]
is the local log canonical threshold at $x$.
\end{defn}

\begin{remark}
    In the literature, the LC-locus is sometimes referred to as the non-klt locus. In a neighborhood of $x \in X$, we have $$\LC_{x}(X,D)\ =\ \NKLT(X,c_x \cdot D),$$ where $c_x=\lct_x(X,D)$.
\end{remark}

\begin{prop}\textup{(Semi-continuity of multiplier ideals, \cite[Corollary 9.5.39]{Lazarsfeld04b})}\label{semicont}
 Let $p:X \rightarrow T$ be a smooth morphism between smooth irreducible varieties. Let $D$ be an effective $\mb{Q}$-divisor on $X$ whose support does not contain any of the fibers $X_t \coloneqq p^{-1}(t)$, so that for each $t\in T$ the restriction $D_t \coloneqq D|_{X_t}$ is well-defined. Suppose moreover that there is a section $q:T\rightarrow X$ of $p$, and write $q_t\coloneqq(t) \in X_t$. If $q_t \in Z(X_t,D_t) $ for all $ 0\neq t \in T$, then $q_0 \in Z(X_0,D_0)$. 
\end{prop}

%



The following well-known lemmas will be used in the proof of the main theorems.

\begin{lemma}\label{lemma:multiplicity_existence} \textup{(\cite[Lemma 10.4.12]{Lazarsfeld04b})}
Let $X$ be an $n$-dimensional irreducible projective variety, and $p_1, \dots, p_m \in X$ be smooth points. Let $H$ be an ample $\Q$-divisor on $X$ such that $(H^n) > m \cdot \gamma^n$ for some rational number $\gamma > 0$. Then there exists an effective $\mb{Q}$-divisor $$D\in\big|H \big|_{\mb{Q}} \quad \textup{with}\quad \mult_{p_{i}}(D)>\gamma \textup{ for each } 1 \le i \le m.$$ 
\end{lemma}

\begin{proof}
This is the multi-point version of \cite[Lemma 10.4.12]{Lazarsfeld04a}, and the proof is almost identical to \cite[Proposition~1.1.31]{Lazarsfeld04a}. Choosing $\ell \gg 0$ sufficiently divisible, we have
\[ h^0(X, \cO_X(\ell H))\ =\  \frac{\ell^n H^n}{n!} + O(\ell^{n-1}), \]
and we need to impose at most $\frac{\gamma^n}{n!} + O(\gamma^{n-1})$ conditions to vanish to order $\ge \gamma$ at each smooth point. Hence, for $\ell \gg 0$ there exist sufficient sections as long as $(H^n) > m \cdot \gamma^n$.
\end{proof}

\begin{lemma}[\textup{\cite[Proposition 9.3.2]{Lazarsfeld04b}}]\label{lemma:multiplicity_multiplier_ideal}
    Let $X$ be a normal projective variety of dimension $n$, and $p$ be a smooth point on $X$. If $D$ is an effective $\mb{Q}$-divisor such that $\mult_p D \geq  n$, then $p$ is contained in $Z(D)$.
\end{lemma}

\subsection{Perturbation, cutting down, and lifting}

From now on, let $(X, H)$ be a normal polarized variety of dimension $n$ such that the canonical divisor $K_X$ is Cartier. 
\par
Let $p_1, \ldots, p_m$ be $m$ distinct points on $X$. Consider a rational number $t_1>0$ such that $(t_1 H)^{n} > m n^{n}$. Then for sufficiently divisible and large $k\in\mb{N}$, by Lemma~\ref{lemma:multiplicity_existence} one can find a divisor $D'\in \big|kt_1 H\big|$ such that
\[ \mult_{p_i}(D')\ >\ kn \quad \textup{ for all } i=1, \ldots, m. \]
As a consequence, setting $D\coloneqq\frac{1}{k}D'$, by Lemma~\ref{lemma:multiplicity_multiplier_ideal} we have that
\[ \mtc{J}\big(X, D\big)_{p_i}\ \subseteq \ \mathfrak{m}_{p_i}\quad \textup{for all }i=1, \ldots, m. \]

\begin{prop}\label{prop:cutting}
Let $X$ be a normal $\bQ$-Gorenstein projective $n$-fold, $p_1,\ldots,p_m \in X$ distinct smooth points, and $D$ an effective $\mb{Q}$-divisor such that $Z(D)$ contains $p_i$ for $i=1, \ldots,m$. For any ample divisor $H$ on $Y$, and any $\epsilon>0$, there exists a rational number $0 \le c < \epsilon$ and an effective $\mb{Q}$-divisor $E\sim_{\mb{Q}}c\cdot H$ such that 
\begin{enumerate}
    \item $\LC_{p_i}(X,D+E)$ is irreducible at $p_i$, and 
    \item $\dim\LC_{p_i}(X,D+E)\leq \dim\LC_{p_i}(X,D)$
\end{enumerate} for every $i=1, \ldots,m$.
\end{prop}

\begin{proof}
Set $Z_i\coloneqq\LC_{p_i}(X,D)$ and $d_i(D)\coloneqq\dim_{p_i}Z_i$ for each $i=1, \ldots,m$. Our proof will proceed by reducing the invariant
\[ s \ \colonequals\  \max \big( \{ 0 \} \cup \{d_i(D) \ |\ Z_{i} \textup{ is reducible} \} \big). \]
If $s = 0$, then we may choose $E = 0$ and the theorem follows. Otherwise, assume that $s>0$ and $d_1(D)$ achieves the maximum $s$. Therefore, $Z_1$ is reducible and $\dim_{p_1}Z_1\geq 1$. Choose an irreducible component $Z_1'$ of $Z_1$ of minimal dimension passing through $p_{1}$, and fix an integer $\ell \gg 1$ such that the sheaf $\mtc{I}_{Z_1'}(\ell H)$ is globally generated. Take a general element $f_1 \in \Gamma(Y,\mtc{I}_{Z_1'}(\ell H))$, and set $E_1\coloneqq\divisor(f_1)$. Then for any rational number $0<\lambda_1\ll1$, the support of $\LC_{p_1}(X, D+\lambda_1 E_1)$ satisfies one of the following conditions:
\begin{enumerate}
    \item it is exactly $Z_1'$, or
    \item it is contained in some proper algebraic subset of $Z_1'$,
\end{enumerate}
because $E_{1}$ only contributes to the log canonical thresholds of the log canonical places whose centers are contained in $Z_{1}'$. Fix an index $i\geq2$. If $p_i\notin Z'_1$, then $\LC_{p_i}(X,D+\lambda_1E_1)=\LC_{p_i}(X,D)$. Now assume that $p_i$ is contained in $Z'_1$. If $\LC_{p_i}(X,D)$ is irreducible, then by choosing $\lambda_1$ sufficiently small, we can guarantee that $\LC_{p_i}(X,D+\lambda_1E_1)=\LC_{p_i}(X,D)$. On the other hand, if $\LC_{p_i}(X,D)$ is reducible, then there are two further cases:
\begin{enumerate}[(a)]
    \item $Z_1'$ is a component of $\LC_{p_i}(X,D)$, in which case $\LC_{p_i}(X,D+\lambda_1E_1)$ is contained in $Z_1'$ (by the first sentence of the previous paragraph), or 
    \item $Z_1'$ is not a component of $\LC_{p_i}(X,D)$, in which case $\LC_{p_i}(X,D+\lambda_1E_1)$ is contained in $\LC_{p_i}(X,D)$ for $0<\lambda_1\ll 1$.
\end{enumerate}
In both cases, we see that $d_{i}(D + \lambda_{1} D_{1}) \leq d_i(D)$.

Replacing $D$ by $D+\lambda_1 E_1$ for some fixed $0<\lambda_1<\frac{\epsilon}{nm}$, and repeating this reduction procedure, after at most $mn$ steps, one will get a $\mb{Q}$-divisor 
\[ E\ \coloneqq\ \sum \lambda_jE_j\ \sim_{\mb{Q}}\ c\cdot H \] 
which satisfies the requisite properties.
\end{proof}

\begin{lemma}[Generic lifting]\label{genericlifting}
    Let $(X,\mtc{O}_X(1))$ be a normal polarized variety, and $Z\subseteq X$ be a closed subvariety. Then for any $\mb{Q}$-divisor $D_Z \in \big| \mtc{O}_Z(d) \big|_{\mb{Q}}$ on $Z$, there exists an effective $\mb{Q}$-divisor $D\in |\mtc{O}_X(d)|_{\mb{Q}}$ on $X$ such that $D|_Z=D_Z$. 
\end{lemma}

\begin{proof}
    Pick a sufficiently large integer $p\gg0$ such that $H^1(X,\mtc{I}_Z(pd))=0$. Then the restriction map $$H^0(X,\mtc{O}_X(pd))\ \longrightarrow \ H^0(Z,\mtc{O}_Z(pd))$$ is surjective. Hence there exists a divisor $D'\in \big|\mtc{O}_X(pd)\big|$ such that $D'|_{Z}=pD_Z$, and the $\mb{Q}$-divisor $\frac{1}{p}D'$ is as desired. If $p$ is sufficiently large and $D'$ is a general divisor satisfying the restriction condition, then we call $D$ a \emph{generic lifting} of $D_Z$.
\end{proof}

\begin{prop}[Cutting down LC-locus]\label{cutting}
    Let $X$ be a normal $\bQ$-Gorenstein projective variety and $p\in X$ be a point. Let $D$ be an effective $\mb{Q}$-divisor on $X$ with $\lct_p(X,D)=1$ such that
    \[ Z\ =\ \LC_p(X,D) \]
    is irreducible of dimension $d$ at $p$. Fix $m$ general smooth points $p_1, \ldots,p_m \in Z$, and let $B$ be any effective $\mb{Q}$-divisor on $X$ with $Z \nsubseteq \Supp B$. Assume that
    \[ \mult_{p_i}(B|_Z) \ >\  d \ =\  \dim Z \]
    for any $i=1, \ldots,m$, where the multiplicity is computed on $Z$. Then for any $0<\epsilon\ll1$, the multiplier ideal
    \[ \mtc{J}(X,(1-\epsilon)D+B) \]
    is non-trivial at $p_i$. If moreover $\mtc{J}(X,D+B)=\mtc{J}(X,D)$ away from $Z$, then in a neighborhood of each $p_{i}$, $Z((1-\epsilon)D+B)$ is a proper algebraic subset of $Z$.
\end{prop}

\begin{proof}
    Following the proof of \cite[Proposition 10.4.10]{Lazarsfeld04b}, we need to choose each $p_{i}$ to avoid a closed subset of $X$ (and we do so for every $i$). Because there are finitely many such points, we can choose all of the $p_{i}$ to avoid a fixed proper subscheme of $X$.
\end{proof}


\section{Separation of points \`a la Angehrn--Siu} \label{section:separation_of_points}

In this section, we will present the proof of Theorem~\ref{thm:AS_type_separation_of_points}. The goal is to construct a $\mb{Q}$-divisor on $X$ by cutting down the LC locus near a point and lifting the divisors, such that the Nadel vanishing theorem applies. This method is motivated by the proof of a weak version of Fujita's conjecture by Angehrn--Siu \cite{AS95}.

\begin{proof}[Proof of Theorem~\ref{thm:AS_type_separation_of_points}]
Let $n \coloneqq \dim{X}$ be the dimension of $X$. First we consider the case where $H$ is very ample. By slicing down to curves, we get a bound $\deg_H W \ge \alpha$ for any positive-dimensional subvariety $W \subset X$ intersecting $U$ from the hypothesis on curves.
For any choice of $m$ distinct points $p_1, \ldots,p_m \in U$, we will construct a $\mb{Q}$-divisor $D$ which satisfies the following conditions:
    \begin{itemize}
        \item $D \sim_{\mb{Q}} c \cdot H$, where $0 < c < d$ is a rational number,
        \item the support of $\mtc{O}_X/\mtc{J}(X,D)$ contains $p_1, \ldots,p_m$, and
        \item some point $p_{j}$ is an isolated point of $\Supp\mtc{O}_X/\mtc{J}(X,D)$.
    \end{itemize}
    
    Supposing that the above holds, we claim that this implies the statement of the theorem. Nadel vanishing \cite[Theorem 9.4.17]{Lazarsfeld04b} gives
    \[ H^1\big(X,\mtc{O}_X(K_X + d H)\otimes \mtc{J}(X, D)\big) \ =\ 0, \]
    and hence the restriction map
    \[ H^0\big(X,\mtc{O}_X(K_X + d H)\big) \longrightarrow H^0\big(X, \mtc{O}_X(K_X + d H)\otimes \mtc{O}_X/\mtc{J}(X,D) \big) \]
    is surjective. Since $\mtc{O}_X/\mtc{J}(X,D)$ has zero-dimensional support near $p_j$ and contains all $p_1, \dots, p_m$, there exists a section of $\mtc{O}_X(K_X + d H)$ whose evaluation at $p_j$ is $1$ and at all other $p_i$ is zero. Removing the point $p_{j}$ from $\{p_1, \ldots, p_m\}$ and repeating the same construction for $\{p_1, \ldots, \hat{p}_{j}, \ldots, p_m \}$, this proves that the evaluation map 
    \[  H^0\big(X,\mtc{O}_X(K_X + d H)\big) \longrightarrow H^0\big(X, \mtc{O}_X(K_X + d H)\otimes \mtc{O}_X / \I_{p_1, \dots, p_m} \big) \]
    is an upper triangular matrix with $1$s along the diagonal when restricted to the span of the divisors constructed above. Hence the evaluation map is surjective.
    \par 
    Now let us begin with the construction of the desired divisor. Choose a positive rational number
    \[ a \ \coloneqq \ \sqrt[^{^n}]{\frac{m n^n}{\alpha}} + \delta \ =\ n\sqrt[^{^n}]{\frac{m}{\alpha}} + \delta, \]
    where $\delta > 0$ is a sufficiently small real number. By Lemma~\ref{lemma:multiplicity_existence}, there exists an effective $\mb{Q}$-divisor $D' \in\big|a H \big|_{\mb{Q}}$ such that $\mult_{p_i}(D') > n$ for each $i=1, \ldots,m$. 
    Recall that our assumption in particular says that $H^n = \deg_H{X} \ge \alpha$. Using Proposition~\ref{prop:cutting}, one can replace $D'$ with a small perturbation $D' + \eta H$ such that $\LC_{p_i}(X,D')$ is irreducible at $p_i$ for each $i$. Setting $\lambda_i \coloneqq \lct_{p_i}(X;D')$, we may reorder
    \[ 0\ <\ \lambda_m\ \leq\  \lambda_{m-1}\ \leq\ \cdots\ \leq\ \lambda_1\ \leq\  1, \]
    where the upper bound comes from the fact that each $p_{i}$ is contained in $Z(D')$ by Lemma~\ref{lemma:multiplicity_multiplier_ideal}.
    Setting $D_{1} \coloneqq \lambda_1 D'$, we have produced a divisor $D_1 \sim_{\Q} a_1 H$, where $a_1 \coloneqq \lambda_1 \left(a + \eta \right)$. Since $\eta$ can be made arbitrarily small, we may define a perturbation $\delta_1$ such that
    \[ a_1 \ = \  n\sqrt[^{^n}]{\frac{m}{\alpha}} + \delta_1, \]
    where $\delta_1$ can be chosen to be $< \epsilon$ for any fixed $\epsilon > 0$. Note that $\delta_{1}$ may be negative, and we are not claiming a lower bound.
    
    By replacing $D'$ with $D_{1}$ as in the previous paragraph, we may assume that $\lambda_1 = 1$. Notice that there could be more than one point at which the log canonical threshold of $D_1$ is equal to $1$. Let $W_1, \ldots,W_v$ be the distinct log canonical loci of $(X,D_1)$ near the points $p_i$ for which we have $\lct_{p_i}(X;D_1)=1$. Reordering the points, we may assume that $p_1, \ldots,p_{\ell}$ are exactly the points with $\lambda_i=1$ and with $W_1$ as the (local) log canonical center. By adding a small multiple of a general divisor in the linear series $\big|\mtc{I}_{W_j}(mH)\big|$ for $j=2, \ldots,v$, we may assume that 
    \[ \lambda_1\ =\ \cdots\ =\ \lambda_{\ell}\ =\ 1\ >\ \lambda_{\ell+1},\  \ldots,\ \lambda_m, \] 
    and $p_1, \ldots,p_{\ell}$ have the same log canonical center $W_1$. 
    \par 
    Let $Z_1 \coloneqq \NKLT(X,D_1)$ be the non-klt locus of $(X, D_1)$, which contains $p_1, \ldots,p_m$ and contains $W_1$ as a component. The dimension of $Z_1$, denoted by $d_2$, is at most $n - 1$. If $\dim_{p_i}{Z_1} = 0$ for some $i$, then we may set $D = D_1$ to construct the necessary divisor $D$. Indeed, in this case $D_1 \sim_{\Q} c H$ where $c \le a_1$, and it is possible to choose $\eta, \delta_1, a_1$ such that $a_1 < d$ if and only if $m > \alpha (d/n)^n$, which holds when $m > \alpha (d - 2n)$. 
    \par 
    Otherwise, we are going to cut down the dimension of the non-klt locus. By assumption, we have $\deg_H{W_1} \ge \alpha$ since $p_1 \in W_1 \cap U$. First assume that $p_1, \ldots,p_m$ are all smooth points of $Z_1$ that are sufficiently general so that Proposition~\ref{cutting} applies. In order to apply the proposition, set $k \coloneqq d_2$ and choose a rational number 
    \[ a_2 \ \coloneqq \  \sqrt[k]{\frac{m k^{k}}{\alpha}} + \delta_2 \ =\  k\sqrt[k]{\frac{m}{\alpha}} + \delta_2 \]
    such that $\delta_2 > 0$ is sufficiently small. Then again by Lemma~\ref{lemma:multiplicity_existence}, one can choose an effective $\mb{Q}$-divisor $D_2'$ on $Z_1$, which is $\mb{Q}$-linearly equivalent to $a_2 H|_{Z_1}$, such that
    \[ \mult_{p_i}D_2 ' \ >\ k \quad \textup{for } i=1, \ldots,\ell. \]
    Let $D_2$ be a generic lift of $D_2'$ to $X$, as in Lemma \ref{genericlifting}. Then $Z((1-t)D_1 + D_2)$ contains $p_1, \ldots,p_m$ for any sufficiently small $t > 0$ by Proposition~\ref{cutting}. 

    If there is a point $p_i$ in the singular locus of $Z_1$, then one can construct $D_2$ using the same limiting process in the proof of Angehrn--Siu's theorem \cite{AS95} (also see \cite[Section 10.4.C, Step 2]{Lazarsfeld04b}). More precisely, let $( 0\in T )$ be a smooth pointed curve and let $u_i : T \rightarrow Z_1$ be a morphism with $u_{i}(0)=p_i$, where $i=1, \ldots,\ell$. Writing $p_{i,t}=u_i(t)\in Z_1$, we may assume that $p_{i,t}$ is a smooth point of $Z_1$ for $t\neq 0$, which is sufficiently general so that we can cut down the LC-locus using Proposition~\ref{cutting}, as in the previous paragraph. Let $D'_{t,2}$ be the divisor that one constructs with respect to points $u_1(t), \ldots,u_l(t)$. We can also assume that for $t \neq 0$, the divisors $D'_t$ vary in a flat family parameterized by $T -\{0\}$. Setting $D'_2 \coloneqq \lim_{t\to 0} D_{t,2}$, we have that $D'_2$ is an effective $\mb{Q}$-divisor on $Z_1$ that lies in the desired linear series and passes through $p_i$ for $i=1, \ldots,\ell$. Taking a generic lift $D_{2}$ of $D'_2$ to $X$, which does not contain $Z_1$ since they intersect properly, we may further assume that
$$\mtc{J}(X,D_1+cD_2)\ =\ \mtc{J}(X,D_1)$$
away from $Z_1$ for every $0 < c < 1$. Now for $t\in T$ near $0$, we may extend $D_2$ to a family of $\mb{Q}$-divisors $D_{t,2}$ which are lifts of $D'_{t,2}$. It follows from Proposition~\ref{cutting} that $$p_{i,t}\ \in\  Z\big((1-t_1)D_1+D_{t,2}\big)$$ for any $0\neq t\in T$ and $0<t_1\ll1$. By the semi-continuity of multiplier ideals (Proposition~\ref{semicont}), a similar containment holds at $p_i=p_{i,0}$:
\[ p_{i}\ \in\  Z\big((1-t_1)D_1+D_2\big). \]

Now that we have produced $D_{1}\sim_{\bQ}a_1H$ and $D_{2}\sim_{\bQ}a_2H$, for all $0 < t_1 \ll 1$ the subscheme $Z((1 - t_1) D_1 + D_2)$ contains $p_1, \dots, p_m$ and is properly contained in $W_1$ near $p_1, \dots, p_{\ell}$.
Now we fix a rational number $0<t_1\ll1$ such that $$\lct_{p_i}(X,(1-t_1)D_1)\ < \ 1 \quad \textup{for all } i > \ell.$$ 
Then there exists a rational number $0 < t_2 < 1$ such that
\begin{itemize}
    \item $\lct_{p_i}\big(X,(1-t_1)D_1+t_2 D_2\big) \leq 1$ for $i=1, \ldots,\ell$;\footnote{This inequality is automatically true for $i=\ell+1, \ldots,m$ since $D_2$ is a generic lift. }
    \item $\max\limits_{1\leq i\leq \ell}\lct_{p_i}\big(X,(1-t_1)D_1+t_2 D_2\big) = 1$, and
    \item $\dim_{p_i} \LC_{p_i}\big(X,(1-t_1)D_1+t_2 D_2\big) < \dim_{p_i} \LC_{p_i}(X,D_1)$, for $i=1, \ldots,\ell$.
\end{itemize}
Applying Lemma~\ref{prop:cutting}, by increasing $t_2$ slightly we may assume that $\LC_{p_i}\big(X,(1-t_1)D_1+t_2 D_2\big)$ is irreducible for any $i=1, \ldots,m$. Let $i=1, \ldots,\ell'$, where $\ell' \leq \ell$, be all the indices such that $$\lct_{p_i}(X,(1-t_1)D_1+t_2 D_2)\ =\ 1.$$ Observe that for any $i=1, \ldots,\ell'$, the scheme $Z((1-t_1)D_1 + t_2 D_2)$ near $p_i$ is an irreducible proper closed subset of $Z_1$, and one has that $$\mtc{J}(X,(1-t_1)D_1 + t_2 D_2)_{p_i}\ \subsetneq \ \mtc{O}_{X,p_i}.$$

Repeating this cutting down and perturbation procedure, we produce divisors $D_j \sim_{\Q} a_j H$. After $s \leq n$ steps, one achieves a $\mb{Q}$-divisor $D\coloneqq(1-t_1)D_1 + t_2 D_2 + \cdots + t_s D_s$ satisfying: 
\begin{itemize}
    \item $0 < t_1, \ldots,t_s < 1$; 
    \item the multiplier ideal sheaf $\mtc{J}(X,D)$ is non-trivial at $p_i$ for any $i$; and
    \item there is an index $i_0\in\{1, \ldots,m\}$ such that $\lct_{p_{i_0}}(X,D)=\max\{\lct_{p_i}(X,D)\}=1$, and $p_{i_0}$ is an isolated point of $Z(D)$.
\end{itemize}
This divisor $D$ satisfies the desiderata as long as $d H - D$ is ample. In fact, we will show that $d H - (D_1 + \cdots + D_s)$ is ample. This is true as long as
\[ d \ >\  a_1 + \cdots + a_s. \]
Consider the sequence $d_1, \dots, d_s$ defined by $d_1 \coloneqq n$ and 
\[ d_{j+1} \ \coloneqq \  \dim{\LC_{p_1}\big(X, (1 - t_1) D_1 + t_2 D_2 + \cdots + t_j D_j\big)} \]
for $2\leq j\leq s$. This sequence is strictly decreasing by construction, and the $a_{j}$ are defined by the formula
\[ a_j \ =\  d_j \sqrt[^{^{d_j}}]{\frac{m}{\alpha}} + \delta_j. \]
The numbers $d_1, \dots, d_s$ form a subsequence of $n, n-1, \dots, 1$, so 
\begin{equation}\label{eq:perturbationBound}
a_1 + \cdots + a_s\ = \ \sum_{j=1}^s \left( d_j \sqrt[^{^{d_j}}]{\frac{m}{\alpha}} + \delta_j \right) \ \le\  \sum_{j = 1}^n j \sqrt[^{^j}]{\frac{m}{\alpha}} + \delta_1 + \cdots + \delta_s.
\end{equation}
When we choose the perturbations $\delta_j$, they can be made arbitrarily small. If we choose $\delta_j$ less than $c(n,d) / n$ for the $c(n,d)$ defined in Lemma~\ref{numerical} (which is positive by the aforementioned lemma) then the inequality \eqref{eq:perturbationBound} is satisfied when $m = \alpha \cdot \lceil d - 2 n \sqrt{d} \rceil$.
\par 
In general, when $H$ is ample but not very ample, we can choose $d_0$ such that $d_0 H$ is very ample. Then by assumption $$\deg_{d_0 H} C \ = \ d_0\cdot (H.\ C) \  \ge\ d_0 \cdot \alpha$$ for any curve $C \subset X$. Hence running the argument for $H$ replaced by $d_0 H$ and $\alpha$ replaced by $d_0 \alpha$, we obtain that $|K_X + d H|$ separates at least $\alpha d_0 \big(d/d_0 - 2 n \sqrt{d/d_0}\big)$ points. Hence the error term $\delta(X, H)$ is increased by a factor of $\sqrt{d_0}$.
\end{proof}

\begin{remark}
    In the statement of Theorem~\ref{thm:AS_type_separation_of_points}, the subset $U \subseteq X^{\reg}$ can be chosen to be the complement of a countable union of proper closed subsets \cite[Proposition 5.17]{Sta17}.
\end{remark}

The following lemma is elementary.

\begin{lemma}\label{numerical}
Fix integers $n > 1$ and $d \ge 4n^2$, and let $m \coloneqq \lceil d - 2n \sqrt{d} \rceil$. Then the quantity
$$    c(n,d) \ \coloneqq\  \frac{1}{2} \bigg( d - \sum_{j = 1}^n j \sqrt[j]{m} \bigg)$$
is strictly positive.
\end{lemma}

Combining Theorem~\ref{thm:AS_type_separation_of_points} with Proposition~\ref{separatepoints} immediately gives:

\begin{theorem}[Covering gonality of divisors] \label{thm:AS_type_gonality}
Let $(X, H)$ be a polarized normal Gorenstein variety. Suppose $\alpha > 0$ is a number and $U \subseteq X^{\reg}$ is a nonempty open set such that any curve $C \subseteq X$ meeting $U$ satisfies $\deg_H{C} \ge \alpha$. Any integral divisor $V \in |d H|$ with at worst canonical singularities that meets $U$ satisfies 
\[ \cg(V)\ \ge\  \big(d - \delta \sqrt{d}\big) \cdot \alpha + 1, \]
where the $\delta$ is the same as in Theorem~\ref{thm:AS_type_separation_of_points}.
\end{theorem}

Note that $V$ is automatically Gorenstein in the set-up above so it makes sense to say that $K_V$ separates many points when $K_X + V$ does. The advantage of allowing an open set $U \subsetneq X$ in the theorem is that, in applications, it suffices to bound from below the degree of a covering family of curves, which is usually easier to estimate than the degree of an arbitrary subvariety. 

\begin{remark} \label{rmk:compareAS}
When $H$ is very ample, by slicing with hyperplanes in $|H|$ we see that the hypothesis on curves is equivalent to assuming that any positive dimensional subvariety $W \subseteq X$ meeting $U$ satisfies $\deg_H{W} \ge \alpha$. For subvarieties $W$ of dimension $> 1$, the hypothesis $\deg_H{W} \ge \alpha$ does not behave linearly when scaling $H$. In \cite{AS95}, the authors assume a hypothesis of the form $\deg_H{W} \ge \alpha^{\dim{W}}$, which does scale linearly.
Directly applying the main theorem of \cite{AS95}, we get a gonality bound of $\covgon(Y) \ge (d/n) \cdot \alpha^{1/n} - O(n)$, which is insufficient to give a positive answer to \cite[Problem 4.1]{BDELU17}.
\par
In fact, Theorem~\ref{thm:degree_of_any_curve} shows that general complete intersections of high degree have the curious property that the minimal degree of a subvariety of dimension $k$ always equals $\deg{X}$, which is independent of $k$, and is achieved by linear slices. Therefore, we cannot extract a better result for applications to measures of irrationality on complete intersections by taking into account higher dimensional subvarieties that might appear in the NKLT locus. 
\end{remark}

Finally, we use Theorem~\ref{thm:AS_type_gonality} to give our desired application to measures of irrationality of complete intersections:

\begin{proof}[Proof of Theorem~\ref{thm:main_gonality}]

Let $Y\subseteq\mb{P}^{n+r}$ be a general complete intersection of type $(d_2, \ldots,d_r)$, and let $X$ be a divisor on $Y$ linearly equivalent to $d_1 H$, where $H$ is the class of a hyperplane section on $Y$. By Proposition~\ref{separatepoints}, it suffices to show that the linear series $\big| K_X \big|$ separates $m$ distinct general points for some integer $m$ satisfying the desired inequality. In fact, we aim to prove a stronger statement: the linear series $\big|K_Y + X\big|$ on $Y$ separates $m$ distinct general points. This follows immediately from Theorem~\ref{thm:AS_type_gonality} applied to the pair $(Y, H)$, setting $d \coloneqq d_{1}$, and Theorem~\ref{thm:main_covering_degree} to bound the degree of any curve of $Y$ outside some proper Zariski closed subset.
\end{proof}

\section{Complements and further questions}

In this section, we present a few interesting applications and state some related questions that arise from our work. 

\subsection{Genus bounds} \label{section:genus_bounds}

As mentioned in the introduction, Theorem~\ref{thm:degree_of_any_curve} implies a lower bound on the geometric genus of any curve in a general complete intersection: 

\begin{theorem}
Let $X \subset \P^{n+r}$ be a very general complete intersection of type $(d_1, \dots, d_r)$ with all $d_i \ge 2n$. Then any curve $C \subset X$ satisfies
\[ p_g(C) \ \ge\  \tfrac{1}{2} (d_1 + \cdots + d_r - 2n - r) (d_1 - 2n + 2) \cdots (d_r - 2n + 2) + 1 .\]
For $d_i \ge N(n,r)$, as in Theorem~\ref{thm:degree_of_any_curve}, one has that 
\[ p_g(C)\ \ge \ \tfrac{1}{2} (d_1 + \cdots + d_r - 2n - r) \cdot d_1 \cdots d_r + 1 \]
\end{theorem}

\begin{proof}
This follows immediately from \cite[Proposition 3.1]{Chen24} and Theorem~\ref{thm:degree_of_any_curve}.
\end{proof}

\noindent When the degrees are large, note that this bound is very close to the genus of a smooth curve coming from slicing by a linear subspace $\Lambda$:
\[ p_g(X \cap \Lambda) \ = \ \tfrac{1}{2} (d_1 + \cdots + d_r - n - r - 1)\cdot  d_1 \cdots d_r + 1.  \]

\subsection{Measures of association} \label{section:association}

A recent paper of Lazarsfeld and Martin \cite{LM23} introduces a measure of complexity between birational classes of varieties of dimension $n$. The starting point for their paper is to consider all \emph{correspondences}
\begin{center}
\begin{tikzcd}
& Z \arrow[ld, swap, "a"] \arrow[rd, "b"]
\\
X & & X'
\end{tikzcd}
\end{center}
where both maps are generically finite and dominant, and $Z$ is a smooth projective variety with $Z \to X \times X'$ birational onto its image. The \textit{correspondence degree} between $X$ and $X'$ is defined as
\[ \mathrm{corr.deg}(X, X') \ \colonequals \ \min_{(Z,a,b)} \{ \deg a \cdot \deg b \}. \]
Then Theorem~\ref{thm:AS_type_separation_of_points} implies a bound on the correspondence degree between two general complete intersections.

\begin{theorem}
Let $X, X' \subset \P^{n+r}$ be two very general complete intersections of type $(d_1, \dots, d_r)$. Assume that all $d_1, \dots, d_r \ge n$. Then 
\[ \mathrm{corr.deg}(X, X')\ \ge\  \bigg( \big(d_1 - 2n \sqrt{d_1}\big) \big(d_2 - n + 1\big) \cdots \big(d_r - n + 1\big) + 1\bigg)^2. \]
\end{theorem}

\begin{proof}
We follow the method of proof for \cite[Theorem A]{LM23}. A correspondence $Z \to X \times X'$ induces a trace map
\[ \tr_Z \ :\  H^0(X, \Omega_X^n) \ \longrightarrow\ H^0(X', \Omega_{X'}^n) \]
given by the map of $\Q$-Hodge structures
\[ Z_* : H^n(X, \mathbb{Q})_{\mathrm{prim}} \ \longrightarrow\ H^n(X', \mathbb{Q}). \]
 For very general $X, X'$, these Hodge structures do not admit any non-zero morphism between each other, and hence the trace map is zero. However, $K_X$ separates at least 
\[ R \ \coloneqq\ \big(d_1 - 2n \sqrt{d_1}\big)\cdot \big(d_2 - n + 1\big) \cdots \big(d_r - n + 1\big) + 1 \]
distinct points on some open set $U \subset X$ by the proof of Theorem~\ref{thm:AS_type_separation_of_points}. As $Z \to X \times X'$ is birational onto its image, the fibers of $Z \to X$ and $Z \to X'$ generically only intersect at one point, so the trace map cannot be zero if $\deg{(Z \to X')} \le R$. Reversing the role of $X$ and $X'$, this proves the claim. 
\end{proof}

\subsection{Generalizations} \label{section:generalizations}

Let $X$ be a smooth projective variety and let $H_1, \dots, H_r$ be \textit{very} ample divisors. Let $H$ be a fixed ample divisor against which we measure degree. For a tuple $\ul{d} = (d_1, \dots, d_r)$, we consider the complete intersection 
\[ X_{\ul{d}} \ =\  \mathbb{V}(s_1, \dots, s_r) \] defined by $s_i \in H^0(X, \struct{X}(d_i H_i))$.
The argument in the proof of Theorem~\ref{thm:degree_of_any_curve} similarly gives:
\begin{theorem}
For a general $X_{\ul{d}}$ with all $d_i \ge n \coloneqq \dim{X} - r$, one has
\[ \covdeg_H(X_{\ul{d}}) \  \ge\  (d_1 - n + 1) \cdots (d_r - n + 1) \cdot \covdeg_H(H_1 \cap \cdots \cap H_r). \]
In particular, one has the following lower bound on $\covgon(X_{\ul{d}})$ as in Theorem~\ref{thm:main_gonality}:
\[  \bigg( \big(d_1 - 2(n+1) \sqrt{d_1}\big)\cdot \big(d_2 - n + 1\big) \cdots \big(d_r - n + 1\big) \bigg)\cdot \covdeg_{H_1}(H_2 \cap \cdots \cap H_r) + 1. \]
\end{theorem}
\noindent
 Of course, fixing a subvariety $X \subset \P^N$, Theorem~\ref{thm:degree_of_any_curve} directly implies a multiplicative degree bound for curves on a general $\O_X(1)$-complete intersection inside $X$. Indeed, those curves are contained in a general complete intersection in $\P^N$. However, this bound involves $N$ and hence is worse than what appears above. Furthermore, the method is able to incorporate intersections of different ample classes as well as to give refined bounds for those curves moving in a covering family (cf.\ Theorem~\ref{thm:main_covering_degree}).

Given these results, one might wonder about more general subvarieties cut out by a large twist of a vector bundle. 

\begin{problem}
Let $(X, \struct{X}(1))$ be a polarized smooth projective variety with $\struct{X}(1)$ very ample. Let $\E$ be a vector bundle on $X$ of rank $r$. For a general section $s \in H^0(X, \E(d))$ for $d \gg 0$ give a lower bound on $\covdeg(V(s), \struct{X}(1))$ in terms of $d$. Is there a bound of order $O(d^r)$?
\end{problem}

\subsection{Sharp bounds}

Several further questions naturally arise from our work. Throughout this section, let $X \subseteq \P^{n+r}$ denote a (very) general complete intersection variety cut out by polynomials of sufficiently large degrees $d_{1}, \ldots, d_{r}$.

\begin{problem}
    Find two constants $d=d(n,r)$ and $N=N(n,r)$ such that if $d_1, \ldots,d_r>d$, then
    \[ \irr(X)\ \geq \ d_1\cdots d_r -N. \]
\end{problem}

Note that for hypersurfaces (where $r = 1$), by \cite[Theorem C]{BDELU17} one can choose the constants $N = 1$ and $d = 2n+1$. The case $n = 1$ of complete intersection curves would follow from \cite{HLU20} and further analysis of secant varieties to general complete intersection curves.

Following conjectures of Griffiths--Harris \cite{GH85} on the divisibility of degrees of curves on very general hypersurfaces, one can more ambitiously ask the following.

\begin{question}
Is there a constant $d=d(n,r)$ such that for $d_1, \ldots,d_r>d$, any curve $C \subseteq X$ satisfies
    \[ d_{1} \cdots d_{r} \divides \deg C\ ? \]
\end{question}

Finally, in light of Theorem~\ref{thm:degree_of_any_curve} it is natural to ask if the only curves of minimal degree are those arising as linear slices.

\begin{question}
Are the minimal degree curves $C \subseteq X$ of the form $C = X \cap \Lambda$ for some linear subspace $\Lambda$ of dimension $r+1$ when $X$ is general of large degree?
\end{question}

\noindent One might even wonder if this is true as long as a general such $X$ does not contain any lines. This would provide a generalization of a theorem of Wu \cite{Wu90}.

Following the discussion after Theorem~\ref{thm:AS_type_separation_of_points}, one can ask what the optimal statement for separation of large numbers of points by adjoint linear series, given bounds on the degrees of curves. Can we do better than an $O(\sqrt{d})$ error term? Is it possible to obtain an $O(1)$ error term as is the case when $X$ is a surface?

\begin{question}
Let $(X, H)$ be a polarized smooth projective variety. Suppose that $\deg_H C \ge \alpha$ for all curves $C \subset X$. Does there exist a constant $\delta \coloneqq \delta(X,H)$ such that $|K_X + d H|$ separates at least $(d - \delta) \cdot \alpha$ distinct points? Can we take $\delta$ to depend only on $\alpha, \mathrm{vol}(H)$, and $\dim{X}$?
\end{question}


\appendix
\section[tocentry={Stable maps to SNC degenerations}]{Stable maps to SNC degenerations}\label{appendix}

\begin{center}
\chapterauthor{Mohan Swaminathan}
\end{center}

In this appendix, we discuss a necessary condition for a stable map in the special fiber of a simple normal crossings (SNC) degeneration to deform to the generic fiber; see Proposition \ref{prop:log_limit} for the precise statement. Variants of this statement are well-known to experts in relative/logarithmic Gromov--Witten theory. Our exposition is self-contained and does not assume any familiarity with relative GW theory or logarithmic geometry.

Proposition \ref{prop:log_limit} is an algebro-geometric version of \cite[Theorem 1.3]{FT22a} and its proof is similar to that of \cite[Theorem 6.1.1]{QChen10}. The statement is reminiscent of Jun Li's \emph{predeformability} condition in relative GW theory \cite{Li01,GV05}, but we avoid the use of \emph{expanded degenerations} of the target by following ideas from log GW theory \cite{GS13,QChen10,AC14,FT22b}. For an SNC degeneration whose total space is regular, the data whose existence is guaranteed by Proposition \ref{prop:log_limit} is very similar (if not identical) to what one needs in order to upgrade an ordinary stable map in the special fiber to a \emph{stable log map over the standard log point}. In view of this, the result of this appendix might be implicitly contained in \cite{ACGS20}.

\subsection{Geometric setup}\label{subsec:appx_geom_setup}

Let $R$ be a DVR which is given by the local ring of a smooth affine curve over $\C$ at a closed point. Let $\m$ be its maximal ideal and $K$ its fraction field. Let $s$ and $\eta$ be the closed point and generic point of $\Delta:=\Spec(R)$ respectively. 

Let $f:\X\to \Delta$ be an SNC degeneration (in the sense of Definition \ref{definition:snc_degeneration}). Write the closed fiber $\X_s = \bigcup_{i\in I}X_i$ of $f$ as the union of finitely many smooth varieties. For any $\emptyset\ne J\subseteq I$, the intersection $\bigcap_{i\in J}X_i$ is either empty or is smooth of codimension $|J|-1$ in $X_i$ for each $i\in J$.

Let $\X^{\reg}\subseteq\X$ be the regular locus. For $i\in I$, define the line bundle $\L_i := \strusheaf{\X^{\reg}}(\X^{\reg}\cap X_i)$ along with its tautological section $\tau_i$ vanishing along $\X^{\reg}\cap X_i$. Let $\tau$ be the tautological section of $\strusheaf{\Delta}(s)$ vanishing at $s$. We have a canonical isomorphism of line bundles
\begin{equation}\label{eqn:snc_degeneration_line_bundles}
    \Psi: f^*(\strusheaf{\Delta}(s))|_{\X^{\reg}} \ \xrightarrow{\sim} \ \textstyle\bigotimes_{i\in I}\L_i
\end{equation}
characterized by mapping $(f^*\tau)|_{\X^{\reg}}$ to $\bigotimes_{i\in I}\tau_i$. Together with the natural $\C$-linear isomorphism $T_{\Delta,s}=(\m/\m^2)^\vee \ \xrightarrow{\sim}\ \strusheaf{\Delta}(s)|_s$, this restricts over $s\in\Delta$ to the trivialization
\begin{align*}
    \Psi_s:T_{\Delta,s}\otimes_{\C}\strusheaf{\X^{\reg}_s} \xrightarrow{\sim} \textstyle\bigotimes_{i\in I}\L_i|_{\X^{\reg}_s}.
\end{align*}
Here $\X^{\reg}_s$ is the fiber over $s$ of $f|_{\X^{\reg}}:\X^{\reg}\to\Delta$.

\indent Let $\mu_s:C\to\X_s$ be a stable map. Define the open subset $U := \mu_s^{-1}(\X^{\reg})\subseteq C$ and let $\Gamma = (V,E)$ be its dual graph, where each vertex $v\in V$ (resp. edge $e\in E$) of $\Gamma$ corresponds to an irreducible component $U_v$ (resp. node $q_e$) of $U$. An edge $e$ joins vertices $v,v'\in V$ exactly when the irreducible components of $U$ on which the node $q_e$ lies are $U_v$ and $U_{v'}$ (note that $e$ is allowed to be a loop, i.e., $v = v'$ is allowed). For $v\in V$, define $\partial U_v = U_v\cap\bigcup_{v'\ne v} U_{v'}$.

\subsection{Main result}\label{subsec:appx_snc_predeformability}

We describe some conditions that must be satisfied by $\mu_s:C\to\X_s$ if it deforms to the generic fiber $\X_\eta$ of $f:\X\to\Delta$ in the sense of Definition~\ref{def:deforms}. The rough idea is to take a $1$-parameter family that deforms $\mu_s$ to the generic fibre and then to record, for each $v \in V$ and $i \in I$, the ``speed" $n_i(v)$ and the ``infinitesimal direction" $\sigma_{v,i}$ at which the component $U_v$ of $U$ ``sinks into" the component $X_i$ of $\X_s$.

\begin{proposition}[cf.\ \cite{FT22a,QChen10}]\label{prop:log_limit}
    If $\mu_s:C\to \X_s$ deforms to $\X_\eta$, then we can find the following data:
    \begin{enumerate}[label = \normalfont(\textbf{\alph*})]
        \item a $\C$-vector space $M$, an integer $n\ge 1$, and an isomorphism $\Phi_\Delta:M^{\otimes n}\xrightarrow{\sim}T_{\Delta,s}$;
        \item a function $\delta:E\to\Z_{\ge 1}$ and isomorphisms $\Phi_e:M^{\otimes \delta(e)}\xrightarrow{\sim}T_{U_v,q_e}\otimes_\C T_{U_{v'},q_e}$ for each edge $e\in E$ joining distinct vertices $v,v'\in V$; and
        \item functions $n_i:V\to\Z_{\ge 0}$ for $i\in I$ and rational sections $\sigma_{v,i}$ of the line bundle 
        \begin{align*}
            \calHom_{\hspace{0.02in} \strusheaf{U_v}} \! \! \big(M^{\otimes n_i(v)}\otimes_\bC \strusheaf{U_v},\mu_s^*\L_i|_{U_v}\big) 
        \end{align*}
        for $v\in V$ and $i\in I$, which are regular and nowhere vanishing on $U_v\setminus\partial U_v$
    \end{enumerate}
    with the following properties.
    \begin{enumerate}[\normalfont(i)]
        \item For $v\in V$ and $i\in I$, $n_i(v) = 0$ if and only if $\mu_s(U_v)\not\subseteq X_i$ and, in this case, $\sigma_{v,i} = (\mu_s^*\tau_i)|_{U_v}$.
        \item For $v\in V$, we have $\sum_{i\in I} n_i(v) = n$ and $\bigotimes_{i\in I}\sigma_{v,i} = (\mu_s^*\Psi_s)\circ(\Phi_\Delta\otimes\normalfont\text{id}_{\strusheaf{U_v}})$.
        \item Whenever $q_e\in \partial U_v$, define $m_i(v,e) := \ord_{q_e}(\sigma_{v,i})\in\Z$ and denote by
        \[
            0\ne \wt\sigma_{v,i}(q_e)\in \Hom_\C\big(T_{U_v,q_e}^{\otimes m_i(v,e)}\otimes_\C M^{\otimes n_i(v)},(\mu_s^*\L_i)|_{q_e}\big)
        \]
        the lowest order term of $\sigma_{v,i}$ at $q_e$. 
        
        For each edge $e\in E$ joining distinct vertices $v,v'\in V$ and $i\in I$, we have
        \[
            n_i(v') = n_i(v) + m_i(v,e)\,\delta(e).
        \]
        In particular, $m_i(v,e) + m_i(v',e) = 0$. Moreover, under the isomorphism
        \[
            T_{U_v,q_e}^{\otimes m_i(v,e)}\otimes_\C M^{\otimes n_i(v)} \ \xrightarrow{\sim} \ T_{U_v',q_e}^{\otimes m_i(v',e)}\otimes_\C M^{\otimes n_i(v')}
        \]
        induced by $\Phi_e$, the lowest order terms $\wt\sigma_{v,i}(q_e)$ and $\wt\sigma_{v',i}(q_e)$ are identified.
    \end{enumerate}
\end{proposition}

\begin{remark}\label{rem:order_sum_zero}
    Property (ii) in Proposition \ref{prop:log_limit} implies that $\bigotimes_{i\in I}\sigma_{v,i}$ is nowhere vanishing. This has a useful numerical consequence: for $v\in V$ and $q_e\in\partial U_v$, we have $\sum_{i\in I}m_i(v,e) = 0$.
\end{remark}

\begin{proof}
    Since $\mu_s:C\to \X_s$ deforms to $\X_\eta$, there is a commutative diagram
    \begin{center}
    \begin{tikzcd}
        \cC \arrow[r,"\mu"] \arrow[d,"g"] & \X \arrow[d,"f"] \\
        B \arrow[r,"\nu"] & \Delta
    \end{tikzcd}    
    \end{center}
    of schemes over $\C$ where
    \begin{itemize}
        \item $B = \Spec(S)$, where $S$ is a DVR given by the local ring of a smooth affine curve over $\C$ at a closed point, with $t,\xi\in B$ being the closed and generic point respectively,
        \item the morphism $\nu: B\to \Delta$ maps $t\in B$ to $s\in \Delta$ and $\xi\in B$ to $\eta\in \Delta$,
        \item $\mu:\cC\to\X$ is a family of stable maps over $B$ whose closed fiber is $\mu_s:C\to\X_s$.
    \end{itemize}
    
    Let $\pi$ be a uniformizer of $S$, let $\mathfrak{n} = (\pi)$ be its maximal ideal and let $L$ be its fraction field. After a base change of DVRs (inducing a finite extension of fraction fields), we ensure that the nodes of the generic fiber $\cC_{\xi}$ are defined over $L$. Each node $\Spec(L)\to\cC_{\xi}$ can be uniquely completed to a section of $\cC\to B$, and this section specializes over $t$ to a node of $C$. The nodes of $C$ obtained in this manner from nodes of $\cC_{\xi}$ will be called \emph{generic nodes} while the remaining nodes of $C$ will be called \emph{special nodes}. 
    
    Define the open subset $\cU := \mu^{-1}(\X^{\reg})\subseteq \cC$ and note that $\cU_t = U$. For convenience, we will denote the restrictions $\mu|_\cU$ and $g|_\cU$ simply by $\mu$ and $g$ respectively.

    With these preparations, we define the desired data as follows.
    \begin{enumerate} [wide, labelwidth=!, labelindent=0pt, label=(\textbf{\alph*})]
        \item Define $M := T_{B,t} = (\mathfrak{n}/\mathfrak{n}^2)^\vee$. Let $\Phi_\Delta:M^{\otimes n}\xrightarrow{\sim} T_{\Delta,s}$ be the isomorphism induced by the ring map $R\to S$ corresponding to $\nu:B\to\Delta$, where $n\ge 1$ is defined by $\m S = \mathfrak{n}^n$.
        \item Let $e\in E$ be an edge joining vertices $v,v'\in V$. The discussion is somewhat different depending on whether $q_e$ is a special node or a generic node.
        
        Suppose $q_e\in U$ is a special node. Then, \'etale locally near $q_e$, the family $\cC\to B$ is given by
        \begin{align}\label{eqn:special_node_model}
            \Spec\left(\frac{S[x,y]}{(xy-\pi^{\delta(e)})}\right)\ \longrightarrow \ \Spec(S),   
        \end{align}
        and this defines $\delta(e)\in\Z_{\ge 1}$.
        We can phrase this intrinsically by saying $\cC$ has an $A_{\delta(e)-1}$ singularity at $q_e$. If $v\ne v'$, then we have a canonical isomorphism $\Phi_e:M^{\otimes\delta(e)}\xrightarrow{\sim} T_{U_v,q_e}\otimes_\C T_{U_{v'},q_e}$. In the \'etale local model \eqref{eqn:special_node_model}, the map $\Phi_e$ is given by $(\partial_\pi)^{\otimes\delta(e)}\mapsto\partial_x\otimes\partial_y$.

        Suppose $q_e\in U$ is a generic node. Then, \'etale locally near $q_e$, the family $\cC\to B$ is given by
        \begin{align}\label{eqn:generic_node_model}
            \Spec\left(\frac{S[x,y]}{(xy)}\right)\ \longrightarrow \ \Spec(S).    
        \end{align}
        In this situation, we define the integer $\delta(e)\in\Z_{\ge 1}$ arbitrarily. When $v\ne v'$, we also define the isomorphism $\Phi_e:M^{\otimes\delta(e)}\xrightarrow{\sim} T_{U_v,q_e}\otimes_\C T_{U_{v'},q_e}$ arbitrarily.
        
        \item Let $i\in I$ and $v\in V$. Consider the section $\sigma_i:=\mu^*\tau_i$ of the line bundle $\mu^*\L_i$ on $\cU$. By assumption, $\sigma_i$ is non-vanishing on the generic fiber $\cU_\xi$ of $\cU\to B$. 
        
        Define $n_i(v)\in\Z_{\ge 0}$ to be the generic order of vanishing of $\sigma_i$ along $U_v$. To explain this in more detail, let $\xi_v$ be the generic point of $U_v$ and note that the local ring $\strusheaf{\cU,\xi_v}$ is a DVR, with uniformizer being the image of $\pi$ under the injective map $S = \strusheaf{B,t}\to \strusheaf{\cU,\xi_v}$ induced by $g$. We can then define $n_i(v)$ to be the valuation of $\sigma_i$ in $\strusheaf{\cU,\xi_v}$, where we use a local trivialization of $\mu^*\L_i$ near $\xi_v$ to regard $\sigma_i$ as an element of $\strusheaf{\cU,\xi_v}$. 
        
        Define $\sigma_{v,i}$ by first regarding $\sigma_i$ as a rational section of the line bundle
        \begin{align*}
            \strusheaf{\cU}(-n_i(v)U)\otimes \mu^*\L_i\ =\ g^*(\strusheaf{B}(-t))^{\otimes{n_i(v)}}\otimes\mu^*\L_i
        \end{align*}
       and then restricting it to $U_v$ using the isomorphism $M = T_{B,t}\xrightarrow{\sim}\strusheaf{B}(t)|_{t}$. Equivalently, we can define $\sigma_{v,i}$ by the explicit formula
        \begin{align}\label{eqn:rescaled_sections_formula}
            \sigma_{v,i} \ =\ (d\pi)^{\otimes n_i(v)}\otimes (\pi^{-n_i(v)}\sigma_i)|_{U_v}.
        \end{align}
        The ideal sheaf of $U_v\setminus\partial U_v\subseteq\cU\setminus\bigcup_{v'\ne v}U_{v'}$ is generated by $\pi$ and the section $\sigma_i$ has no zeros on $\cU_\xi$. It follows that $\sigma_{v,i}$ is regular and nowhere vanishing on $U_v\setminus\partial U_v$.
    \end{enumerate}

    We now check that the above definitions have all the claimed properties.

    \begin{enumerate} [wide, labelwidth=!, labelindent=0pt, label=(\roman*)]
        \item By definition, we have $n_i(v) = 0$ if and only if $\sigma_i$ is non-vanishing at the generic point $\xi_v$ of $U_v$, i.e., $\mu_s(U_v)\not\subseteq X_i$. In this case, \eqref{eqn:rescaled_sections_formula} implies that $\sigma_{v,i} = \sigma_i|_{U_v} = (\mu_s^*\tau_i)|_{U_v}$.
        
        \item The valuation of the section $\bigotimes_{i\in I}\sigma_i$ in the DVR $\strusheaf{\cU,\xi_v}$ is $\sum_{i\in I} n_i(v)$. Under the pullback of the isomorphism \eqref{eqn:snc_degeneration_line_bundles} under $\mu$, this section becomes $\mu^*(f^*\tau) = g^*(\nu^*\tau)$. Since $\pi$ is simultaneously a uniformizer for both the DVRs $S$ and $\strusheaf{\cU,\xi_v}$, we conclude that $\sum_{i\in I}n_i(v)$ is the same as the valuation of $\nu^*\tau$ in $S$, which is $n$.  
        
        Regard $\nu^*\tau$ as a section of the line bundle
        \begin{align*}
            \strusheaf{B}(-nt)\otimes\nu^*(\strusheaf{\Delta}(s)).
        \end{align*}
        Then, its image under the isomorphism $(\strusheaf{B}(-nt)\otimes\varphi^*(\strusheaf{\Delta}(s)))|_{t}\xrightarrow{\sim}\Hom_\bC(T_{B,t}^{\otimes n},T_{\Delta,s})$ coincides with $\Phi_\Delta$. The claimed formula for $\bigotimes_{i\in I}\sigma_{v,i}$ now follows from the fact that the section $(f^*\tau)|_{\X^{\reg}}$ maps to the section $\bigotimes_{i\in I}\tau_i$ under the isomorphism \eqref{eqn:snc_degeneration_line_bundles}.  
        
        \item Fix $i\in I$ and let $e$ be an edge joining $v\ne v'$. Using a local trivialization of $\mu^*\L_i$ near $q_e$, we regard $\sigma_i = \mu^*\tau_i$ as a regular function defined on a neighborhood of $q_e$ in $\cU$. Since the assertions to be checked are \'etale local, we will work in the local models \eqref{eqn:special_node_model} and \eqref{eqn:generic_node_model}. In terms of these local models, assume that $U_v$ is given by $\pi = 0$, $x = 0$ and that $U_{v'}$ is given by $\pi = 0$, $y = 0$.
        
        Suppose $q_e$ is a special node. Since $\sigma_i$ has no zeros on $\cU_\xi$, we can write $\sigma_i = \pi^a x^b y^ch$ for some $a,b,c\in\Z_{\ge 0}$ and $h\in\strusheaf{\cU}^\times$. Since $y\in\strusheaf{\cU,\xi_v}^\times$, we can write $x = \pi^{\delta(e)}y^{-1}\in\strusheaf{\cU,\xi_v}$ which gives $\sigma_i = \pi^{a + \delta(e)b}y^{c-b}h$. From this, we deduce that $n_i(v) = a + \delta(e)b$,
        \begin{align*}
            \sigma_{v,i} = (d\pi)^{\otimes(a + \delta(e)b)}\otimes (y^{c-b}h)|_{\pi = 0,\,x = 0},
        \end{align*}
        and $m_i(v,e) = c - b$. Similarly, we get $n_i(v') = a + \delta(e)c$ and $m_i(v',e) = b - c$, which shows that $n_i(v') = n_i(v) + m_i(v,e)\delta(e)$. Finally, note that the lowest order terms
        \begin{align*}
            \wt\sigma_{v,i}(q_e) &= (d\pi)^{\otimes(a + \delta(e)b)}\otimes (dy)^{\otimes(c-b)}\otimes h(q_e),\\
            \wt\sigma_{v',i}(q_e) &= (d\pi)^{\otimes(a + \delta(e)c)}\otimes (dx)^{\otimes(b-c)}\otimes h(q_e)    
        \end{align*}
        coincide under the isomorphism $\Phi_e$, which identifies $\partial_x\otimes\partial_y$ with $(\partial_\pi)^{\otimes\delta(e)}$.

        Suppose $q_e$ is a generic node. Let $\wt q_e: B\to\cC$ be the completion of the node $\Spec(L)\to\cC_\xi$ which specializes to $q_e$ over $t\in B$. The section $\wt q_e$ corresponds to the section of \eqref{eqn:generic_node_model} given by 
        \begin{align*}
            \Spec(S)\ \longrightarrow \ \Spec\left(\frac{S[x,y]}{(xy)}\right),\quad \ x,y\mapsto 0.
        \end{align*}
        Normalize an \'etale neighborhood of the section $\wt q_e$ in $\cU$ to break it into two components $\cU_v$ and $\cU_{v'}$ such that $(\cU_v)|_t = U_v$ and $(\cU_{v'})|_t = U_{v'}$. The lifts of the section $\wt q_e$ to $\cU_v$, $\cU_{v'}$ are given in the local model by
        \begin{align*}
            \Spec(S)\ \longrightarrow \ \Spec\left(\frac{S[x,y]}{(x)}\right) = \Spec\left(S[y]\right),\quad \ y\mapsto 0,\\
            \Spec(S)\ \longrightarrow\  \Spec\left(\frac{S[x,y]}{(y)}\right) = \Spec\left(S[x]\right),\quad \ x\mapsto 0.
        \end{align*}
        respectively.
        Since $\sigma_i$ has no zeros on $\cU_\xi$, we can find $h_v\in \strusheaf{\cU_v}^\times$ and $h_{v'}\in\strusheaf{\cU_{v'}}^\times$ with
        \begin{align*}
            \sigma_i|_{\cU_v} = \pi^{n_i(v)}h_v\quad\text{and}\quad
            \sigma_i|_{\cU_{v'}} = \pi^{n_i(v')}h_{v'}.
        \end{align*}
        Since these two expressions for $\sigma_i$ must agree on the section $\wt q_e$, we must have $n_i(v) = n_i(v')$, $h_v(q_e) = h_{v'}(q_e)\ne 0$ and $m_i(v,e) = m_i(v',e) = 0$. Thus, $n_i(v') = n_i(v) + m_i(v,e)\delta(e)$ and 
        \begin{align*}
            \wt\sigma_{v,i}(q_e) = \sigma_{v,i}(q_e) = \sigma_{v',i}(q_e) = \wt\sigma_{v',i}(q_e).
        \end{align*}
    \end{enumerate}
    This concludes the proof.
\end{proof}

\subsection{Deducing Lemma \ref{lemma:divisorial_multiplicity_matching} from Proposition \ref{prop:log_limit}} \label{subsec:appx_deducing_lemma}

Restrict to the case where $\X_s = X_1\cup_Z X_2$ is the union of two smooth varieties $X_1$ and $X_2$ meeting along a smooth divisor $Z = X_1\cap X_2$.

Recall from the statement of Lemma \ref{lemma:divisorial_multiplicity_matching} that $F\subseteq C$ is a connected component of $\mu_s^{-1}(Z)$ which is contracted to a point $z = \mu_s(F)\in \X^{\reg}$. Moreover, for $i=1,2$, we define $C_i$ to be the union of the irreducible components of $C$ which map under $\mu_s$ to $X_i$ (but not $Z$) and meet $F$. Given that the non-constant stable map $\mu_s:C\to\X_s$ deforms to $\X_\eta$, we need to show that
\begin{align*}
    \sum_{p \in F \cap C_1} m_p(C_1 ; Z) \ =\ \sum_{p \in F \cap C_2} m_p(C_2 ; Z),
\end{align*}
where $m_p(C_i ; Z)$ is the multiplicity at $p$ of the divisor on $C_i$ obtained by pulling back $Z$ under the map $\mu_s|_{C_i}:C_i\to X_i$ for $i=1,2$.

Apply Proposition \ref{prop:log_limit} to produce the data of (a), (b), (c) satisfying properties (i), (ii), (iii). As before, write $\Gamma = (V,E)$ for the dual graph of $U = \mu_s^{-1}(\X^{\reg})$, so that $F\subseteq U$. Fixing $i=1,2$, note that each irreducible component of $C_i$ meets $U$. Let $V_i\subseteq V$ be the set of vertices corresponding to the irreducible components of $U\cap C_i$. Since the irreducible components of $C_i$ and $U\cap C_i$ are in bijection, we will write $C_{v,i}$ in place of $U_v$ when $v\in V_i$.

\begin{lemma}\label{lemma:appx_auxiliary_observations}
    Consider $p\in F\cap C_{v,1}$ for some $v\in V_1$. Then, we have $n_2(v) = 0$, $\sigma_{v,2} = (\mu_s^*\tau_2)|_{C_{v,1}}$ and $\ord_p(\sigma_{v,2}) = m_p(C_1;Z)$. Moreover, there exists $v'\in V\setminus V_1$ such that $p\in U_{v'}$.
\end{lemma}
\begin{proof}
    Since $\mu_s(C_{v,1})\not\subseteq X_2$, property (i) in Proposition \ref{prop:log_limit} gives $n_2(v) = 0$ and $\sigma_{v,2} = (\mu_s^*\tau_2)|_{C_{v,1}}$. The restrictions of $\cL_2 = \strusheaf{\X^{\reg}}(\X^{\reg}\cap X_2)$ and $\strusheaf{X_1}(Z)$ to $\X^{\reg}\cap X_1$ are canonically isomorphic and their tautological sections get identified under this isomorphism. It follows that we have $\ord_p(\sigma_{v,2}) = m_p(C_1;Z)$.

    Since $p\in F$, the section $\sigma_{v,2}$ vanishes at $p$. Since $\sigma_{v,2}$ is non-vanishing away from $\partial C_{v,1}$, the point $p$ must lie on $U_{v'}$ for some $v\ne v'\in V$. If we had $v'\in V_1$, then the preceding paragraph shows that $\sigma_{v',2}$ is regular (and vanishes) at $p$, contradicting property (iii) of Proposition \ref{prop:log_limit}, which gives $\ord_p(\sigma_{v,2}) + \ord_p(\sigma_{v',2}) = 0$. Thus, $v'\in V\setminus V_1$.
\end{proof}

If $F$ is $0$-dimensional, then write $F = \{p\}$. Since $p$ lies on at most two irreducible components of $C$, we must have $|V_1| + |V_2|\le 2$. As $\mu_s$ is non-constant, at least one of $V_1$ and $V_2$ is non-empty. Without loss of generality, assume $V_1\ne\emptyset$ and pick $v\in V_1$. By Lemma \ref{lemma:appx_auxiliary_observations}, we get $v'\in V\setminus V_1$ such that $p\in U_{v'}$. Since $F = \{p\}$ is a connected component of $\mu_s^{-1}(Z)$, we must have $v'\in V_2$. This shows $|V_1| = |V_2| = 1$. By property (iii) of Proposition \ref{prop:log_limit}, Remark \ref{rem:order_sum_zero} and Lemma \ref{lemma:appx_auxiliary_observations}, we get
\begin{align*}
    0 \ =\ \ord_{p}(\sigma_{v,2}) + \ord_{p}(\sigma_{v',2}) \ =\ \ord_{p}(\sigma_{v,2}) - \ord_{p}(\sigma_{v',1})\ =\ m_p(C_1;Z) - m_p(C_2;Z).
\end{align*}

If $F$ is $1$-dimensional, let $V_0\subseteq V$ be set of the vertices corresponding to the irreducible components of $U\cap F = F$. For $v\in V_0$, we will write $F_v$ in place of $U_v$. Since $\mu_s(F)$ is a point in $\X^{\reg}\cap Z$, the bundles $\mu_s^*\L_i$ are trivial on $F_v$ for all $v\in V_0$ and $i=1,2$. Since $F_v$ is a proper curve and $\sigma_{v,i}$ is regular and non-vanishing at all nodes of $F_v$, we conclude that $\sigma_{v,i}$ must have as many zeros as it has poles (counted with multiplicity) for all $v\in V_0$ and $i=1,2$. Fixing $i=1$ and taking the sum over all $v\in V_0$, we get the relation
\begin{align*}
    \sum_{v\in V_0}\sum_{q_e\in\partial F_v}\ord_{q_e}(\sigma_{v,1}) = 0.
\end{align*}
For distinct $v,v'\in V_0$ and $q_e\in F_v\cap F_{v'}$, we have $\ord_{q_e}(\sigma_{v,1}) + \ord_{q_e}(\sigma_{v',1}) = 0$ by property (iii) of Proposition \ref{prop:log_limit}. Such terms therefore cancel in pairs and the above relation reduces to
\begin{align*}
    \sum_{v\in V_0}\sum_{p\in F_v\cap C_1}\ord_{p}(\sigma_{v,1}) + \sum_{v\in V_0}\sum_{p\in F_v\cap C_2} \ord_{p}(\sigma_{v,1}) = 0.
\end{align*}
Using Remark~\ref{rem:order_sum_zero}, we can rewrite this as
\begin{align*}
    -\sum_{v\in V_0}\sum_{p\in F_v\cap C_1}\ord_p(\sigma_{v,2}) + \sum_{v\in V_0}\sum_{p\in F_v\cap C_2}\ord_p(\sigma_{v,1}) = 0.
\end{align*}
Using property (iii) in Proposition \ref{prop:log_limit}, this becomes
\begin{align*}
    \sum_{v\in V_1}\sum_{p\in F\cap C_{v,1}}\ord_p(\sigma_{v,2}) - \sum_{v\in V_2}\sum_{p\in F\cap C_{v,2}}\ord_p(\sigma_{v,1}) = 0.
\end{align*}
By Lemma \ref{lemma:appx_auxiliary_observations}, $\ord_p(\sigma_{v,2}) = m_p(C_1;Z)$ for all $v\in V_1$ and $p\in F\cap C_{v,1}$. Similarly, for all $v\in V_2$ and $p\in F\cap C_{v,2}$, we have $\ord_p(\sigma_{v,1}) = m_p(C_2;Z)$. This completes the proof of Lemma~\ref{lemma:divisorial_multiplicity_matching}.


\bibliographystyle{alpha}
\bibliography{biblio.bib}

\footnotesize{
\textsc{Department of Mathematics, Harvard University, Massachusetts 02138} \\
\indent \textit{E-mail address:} \href{mailto:nathanchen@math.harvard.edu}{nathanchen@math.harvard.edu}

\textsc{Department of Mathematics, Stanford University, California 94305} \\
\indent \textit{E-mail address:} \href{mailto:bvchurch@stanford.edu}{bvchurch@stanford.edu}

\textsc{Department of Mathematics, University of Maryland, College Park, Maryland 20742} \\
\indent \textit{E-mail address:} \href{mailto:jzhao81@umd.edu}{jzhao81@umd.edu}

\textsc{Department of Mathematics, Stanford University, California 94305} \\
\indent \textit{E-mail address:} \href{mailto:mohans@stanford.edu}{mohans@stanford.edu}
}

\end{document}